\documentclass[a4paper, 10pt]{article}

\usepackage{packageList}
\usepackage{macros}

\graphicspath{{./Figures/}}
\pgfplotsset{compat=1.11}
\newlength\fwidth

\title{Sharp inverse statements for kernel interpolation}

\author[1,2]{Tizian Wenzel \thanks{wenzel@math.lmu.de}}
\affil[1]{Institute for Applied Analysis and Numerical Simulation, University of Stuttgart, Germany}
\affil[2]{Department of Mathematics, Ludwig-Maximilians-Universität München, Germany}

\DeclareMathOperator{\vol}{vol}

\newcommand{\Ht}{\mathcal{H}_{\vartheta}}
\newcommand{\HT}{\mathcal{H}_{\Theta}}

\newtheorem{conj}{Conjecture}

\newcommand{\tw}[1]{\textcolor{magenta}{TW: #1}}

\usepackage{comment}
\newif\iflong			
\longfalse

\newif\ifappendix
\appendixfalse

\usepackage{comment}

\usepackage[normalem]{ulem}			%
\usepackage{thm-restate}

\makeatletter
\newcommand*\bigcdot{\mathpalette\bigcdot@{.5}}
\newcommand*\bigcdot@[2]{\mathbin{\vcenter{\hbox{\scalebox{#2}{$\m@th#1\bullet$}}}}}
\makeatother

\begin{document}

\maketitle %

\begin{abstract}
While direct statements for kernel based interpolation on regions $\Omega \subset \R^d$ are well researched, 
far less is known about corresponding inverse statements.
The available inverse statements for kernel based interpolation so far are not sharp.

In this paper,
we derive sharp inverse statements for interpolation using finitely smooth kernels,
such as popular radial basis function (RBF) kernels like the class of Matérn or Wendland kernels.
In particular, the results show that there is a one-to-one correspondence between the smoothness of a function and its approximation rate via kernel interpolation: 
If a function can be approximated with a given rate, it has a corresponding smoothness and vice versa. \newline

\textbf{MSC classification:} 65D05, 46E22, 41A27, 41A17

\textbf{Keywords:} Kernel interpolation $\cdot$ Inverse estimates $\cdot$ Radial basis functions
\end{abstract}

\section{Introduction} \label{sec:introduction}

In approximation theory, direct statements give error estimates for the approximation of an element (e.g.\ a function) %
using previously defined basis elements, e.g.\ polynomials.
The derived convergence rate usually depends on properties of the element to be approximated, e.g.\ the smoothness of the function.
Inverse statements address the contrary and derive properties of the element which is approximated (e.g.\ the smoothness of the function) based on the approximation rate.

This paper deals with the particular case of function approximation via kernel interpolation, 
see e.g.~\cite{wendland2005scattered, fasshauer2007meshfree, fasshauer2015kernel, steinwart2008support} for a general introduction to kernel methods:
Given a nonempty set $\Omega$, a kernel is defined as a symmetric function $k: \Omega \times \Omega \rightarrow \R$.
While the set $\Omega$ can be in general quite arbitrary, we are interested in the frequent case of compact regions $\Omega \subset \R^d$. \\
For a given set of $n \in \N$ points $X_n \subset \Omega$, the kernel matrix $A_{X_n}$ is defined as $(A_{X_n})_{ij} = (k(x_i, x_j))_{ij} \in \R^{n \times n}$.
If the kernel matrix is positive definite for any $n$ pairwise distinct points $X_n \subset \Omega$ and any $n \in \N$, then the kernel is called strictly positive definite.
For such strictly positive definite kernels, there exists a unique associated Hilbert space $\ns$ of real-valued functions, the so called \textit{reproducing kernel Hilbert space} (RKHS), 
which is also called \textit{native space}.
In this space the kernel $k$ acts as a reproducing kernel, i.e.\ 
\begin{enumerate}
\item $k(\cdot, x) \in \ns ~ \forall x \in \Omega$,
\item $f(x) = \langle f, k(\cdot, x) \rangle_{\ns} ~ \forall x \in \Omega,  \forall f \in \ns$ \hfill (reproducing property).
\end{enumerate}
Kernels can be used for various tasks, and here we are concerned with the approximation of continuous functions by interpolation using continuous strictly positive definite kernels:
Given a continuous function $f \in \mathcal{C}(\Omega)$ and pairwise distinct interpolation points $X_n \subset \Omega$, we seek an interpolating function $s_{f, X_n}$ such that
\begin{align}
\label{eq:interpolation_conditions}
s_{f, X_n}(x_i) = f(x_i) \quad \forall x_i \in X_n.
\end{align}
A well known representer theorem shows
that there exists a minimum norm interpolant in $\ns$ of the form
\begin{align*}
s_{f, X_n}(\cdot) = \sum_{j=1}^n \alpha_j^{(n)} k(\cdot, x_j) \in \ns.
\end{align*}
The coefficients $\alpha_j^{(n)} \in \R, j= 1, \dots,n$ can be calculated by solving the linear equation system arising from Eq.~\eqref{eq:interpolation_conditions}, 
which is always solvable due to the assumption on the strict positive definiteness of the kernel.
Furthermore,
based on the reproducing property, it is easy to see that it holds $\Vert s_{f, X_n} \Vert_{\ns}^2 = (\alpha^{(n)})^\top A_{X_n} \alpha^{(n)}$.

Under suitable conditions on the kernel $k$ and on the compact region $\Omega \subset \R^d$ (such as a Lipschitz boundary), 
the RKHS can be characterized in terms of Sobolev spaces, i.e.\ there is a norm-equivalence $\mathcal{H}_k(\Omega) \asymp H^\tau(\Omega)$ for some $\tau > d/2$.
This is frequently the case, e.g.\ for radial basis function (RBF) kernels such as the popular classes of Matérn or Wendland kernels,
and will be the focus of the current paper.

In this case of $\ns \asymp H^\tau(\Omega)$ for some $\tau > d/2$, 
the approximation error between the function $f$ and the interpolant $s_{f, X}$ is usually bounded with help of the fill distance $h_{X, \Omega}$ and uniformity constant $\rho_X$ (defined in Eq.~\eqref{eq:fill_sep_dist} and Eq.~\eqref{eq:uniformly_distributed_points} respectively).
For $f \in H^\beta(\Omega) \supseteq H^\tau(\Omega)$ with $d/2 < \beta \leq \tau$, 
these estimates (see \Cref{th:error_estimate} for a precise statement) read
\begin{align}
\label{eq:direct_statement}
\Vert f - s_{f,X} \Vert_{L^2(\Omega)} \leq C h_{X, \Omega}^\beta \, \rho_{X}^{\tau-\beta} \Vert f \Vert_{H^\beta(\Omega)},
\end{align}
providing a convergence rate of $\beta$ in the fill distance $h_{X, \Omega}$ for uniformly distributed points $X \subset \Omega$.
This convergence rate matches the smoothness of the function $f \in H^\beta(\Omega)$.

In the current work, we prove corresponding inverse statements, 
which allow to derive the smoothness of a function $f \in C(\Omega)$ based on the rate of approximation. 
To be precise, we have the following main result of this paper, which will be proven in \Cref{sec:inverse_statement}:

\begin{restatable}{theorem}{mainresult}[Main result]
\label{th:inverse_statement_generalized}
Consider a compact Lipschitz region $\Omega \subset \R^d$ and a continuous kernel $k$ such that $\ns \asymp H^\tau(\Omega)$ for some $\tau > d/2$.
Consider $f \in \mathcal{C}(\Omega)$ such that for some $c_f, \beta, h_0 > 0$ it holds
\begin{align}
\label{eq:inverse_statement_assumption}
\Vert f - s_{f, X} \Vert_{L^2(\Omega)} \leq c_f h_{X, \Omega}^\beta
\end{align}
for all $X \subset \Omega$ with $h_{X, \Omega} \leq h_0$ and $\rho_X \leq 44$ (quasi-uniformly distributed). \\
If $\beta \in (0, \tau]$,
then $f \in H^{\beta'}(\Omega)$ for all $\beta' \in (0, \beta)$.
If $\beta > \tau$, then $f \in \ns \asymp H^\tau(\Omega)$. 
\end{restatable}

This result significantly extends the current state of the art compared to previously known inverse statements for kernel interpolation, 
see \Cref{subsec:related_work} and \Cref{sec:discussion_remarks} for a detailed discussion.
In particular this inverse statement \Cref{th:inverse_statement_generalized} is sharp, 
because it provides the correct smoothness in view of the corresponding direct statement Eq.~\eqref{eq:direct_statement} (up to an arbitrary small $\varepsilon > 0$).
Thus we obtain for the first time a one-to-one correspondence between the smoothness of a function and the rate of approximation for kernel interpolation.

The remaining paper is structured as follows: 
\Cref{sec:background} collects necessary background information on the tools which will be used. 
\Cref{sec:utility_statements} provides utility statements, 
which will then be leveraged in \Cref{sec:inverse_statement} to prove the main result \Cref{th:inverse_statement_generalized}.
\Cref{sec:discussion_remarks} discusses the results and elaborates on open directions for future research,
while \Cref{sec:conclusion} concludes the paper.

\section{Background}
\label{sec:background}

In the following, we collect some further required background information on geometric properties of $\Omega$ (\Cref{subsec:geometric_properties}) and Sobolev spaces (\Cref{subsec:sobolev_spaces}), 
kernels and kernel interpolation (\Cref{subsec:kernel_approx}) 
and review related literature (\Cref{subsec:related_work}).

\subsection{Geometric properties related to $\Omega$} \label{subsec:geometric_properties}

A standard assumption on the regularity of the region $\Omega \subset \R^d$, 
which is also frequently used in the following, is given by an interior cone condition.
We note that a Lipschitz boundary already implies the interior cone condition. 
\begin{definition}
A set $\Omega \subset \R^d$ is said to satisfy an interior cone condition, if there exists an angle $\alpha \in (0, \pi/2)$ and a radius $r>0$ such that for every $x \in \Omega$ a unit vector $\xi(x)$ exists such that the cone 
\begin{align}
\label{eq:definition_cone}
C(x, \xi(x), \alpha, r) := \{ x + \lambda y: y \in \R^d, \Vert y \Vert_2 = 1, y^\top \xi(x) \geq \cos(\alpha), \lambda \in [0, r] \}
\end{align}
is contained in $\Omega$.
\end{definition}

In order to quantify the distribution of finitely many scattered points $X \subset \Omega$ in a region $\Omega \subset \R^d$, 
we recall the definition of separation distance $q_{X}$ and fill distance $h_{X, \Omega}$ as
\begin{align}
\label{eq:fill_sep_dist}
\begin{aligned}
q_{X} :=& \frac{1}{2} \min_{x_i \neq x_j \in X} \Vert x_i - x_j \Vert_2, \\
h_X := h_{X, \Omega} :=& \sup_{x \in \Omega} \min_{x_j \in X} \Vert x - x_j \Vert_2.
\end{aligned}
\end{align}
We call a sequence $(X_n)_{n \in \N}$ of sets of points $X_n \subset \Omega$ quasi-uniformly distributed, 
if the uniformity constant $\rho_{X_n} := \rho_{X_n, \Omega} := h_{X_n, \Omega} / q_{X_n}$ is uniformly bounded for all $n \in N$, i.e.\
\begin{align}
\label{eq:uniformly_distributed_points}
\exists C>0 ~ \forall n \in \N \quad \rho_{X_n} \equiv \frac{h_{X_n, \Omega}}{q_{X_n}} \leq C < \infty.
\end{align}
Also a sequence of points $( x_n )_{n \in \N}$ is called quasi-uniformly distributed, 
if the sequence $( \{x_j\}_{j=1}^n )_{n \in \N}$ is quasi-uniformly distributed.

Let $C_d$ denote the volume of the unit ball $B_1(0) \subset \R^d$ and let $C_{d, \alpha}$ be the volume of the unit cone $C(x, \xi(x), \alpha, 1)$ (which is independent of $x \in \Omega$), i.e.\
\begin{align}
\label{eq:constants_ball_cone}
\begin{aligned}
C_d &:= \vol(B_1(0)), \\
C_{d, \alpha} &:= \vol(C(x, \xi(x), \alpha, 1)).
\end{aligned}
\end{align}
Given a bounded region $\Omega \subset \R^d$ which satisfies an interior cone condition, 
and $n \in \N$ points $X_n \subset \Omega$, 
it holds
\begin{align}
\label{eq:geometric_bounds}
\begin{aligned}
h_{X_n, \Omega} &\geq c_{\Omega} \cdot n^{-1/d}, \\
q_{X_n} &\leq C_{\Omega} \cdot n^{-1/d},
\end{aligned}
\end{align}
with constants
\begin{align*}
c_\Omega &= \frac{1}{\pi^{1/2}} \cdot \left( \vol(\Omega) \Gamma \left( \frac{d}{2} + 1 \right) \right)^{1/d}, \\
C_\Omega &= \left( \frac{2\pi}{\alpha} \right)^{1/d} \cdot \frac{1}{\pi^{1/2}} \cdot \left( \vol(\Omega) \Gamma \left( \frac{d}{2} + 1 \right) \right)^{1/d},
\end{align*}
(see \cite[Satz 2.1.6]{mueller2009komplexitaet} and \cite[Satz 2.1.7]{mueller2009komplexitaet}), especially it holds $c_\Omega < C_\Omega$. 
In the following,
we use the shorthand notation $h_X := h_{X, \Omega}$, as the region $\Omega$ is usually fixed and clear from the context.

A convenient way to obtain well distributed points for bounded regions $\Omega \subset \R^d$ is to use the \textit{geometric greedy algorithm}, 
which was introduced and analyzed in \cite{marchi2005optimal}.
We waive to recall the geometric greedy algorithm in itself, we simply recall that the resulting sequence of points satisfies for all $n \geq 2$ the convenient property \cite[Lemma 5.1]{marchi2005optimal}
\begin{align}
\label{eq:geometric_uniformity}
\frac{1}{2} h_{X_n, \Omega} \leq \frac{1}{2} h_{X_{n-1}, \Omega} \leq q_{X_n} \leq h_{X_n, \Omega}.
\end{align}
The constructions within the next sections will frequently make use of these points:
As soon as a region $\Omega$ is fixed, 
we consider \textit{the} sequence of geometric greedy points, 
though it is clear that there can be several such sequences, e.g.\ depending on the choice of the first point.
However, this will be used as notion to enforce that we think all the time about the same sequence of points.

\subsection{Sobolev spaces}
\label{subsec:sobolev_spaces}

We use standard notation \cite{agranovich2015sobolev}: 
The (fractional) Sobolev spaces $H^\tau(\R^d)$ can be defined for also non-integer $\tau \geq 0$ with help of the decay of the Fourier transform $\hat{u}$ of a function $u \in L^2(\R^d)$ as
\begin{align}
\begin{aligned}
\label{eq:definition_sobolev_Rd}
H^\tau(\R^d) &= \{ u \in L^2(\R^d) ~ | ~ \hat{u}(\cdot)(1 + \Vert \cdot \Vert_2^2)^{\tau/2} \in L^2(\R^d) \}, \\
\Vert u \Vert_{H^\tau(\R^d)}^2 &= \int_{\R^d} (1+\Vert \omega \Vert_2^2)^{\tau} \cdot |\hat{u}(\omega)|^2 ~ \mathrm{d}\omega.
\end{aligned}
\end{align}
For regions $\Omega \subset \R^d$, the localized Sobolev spaces $H^\tau(\Omega)$ can be defined via restrictions as
\begin{align*}
H^\tau(\Omega) &= \{ u|_\Omega ~ | ~ u \in H^\tau(\R^d) \}, \\
\Vert u \Vert_{H^\tau(\Omega)} &= \inf \{\Vert v \Vert_{H^\tau(\R^d)} ~ | ~ v \in H^\tau(\R^d), v|_\Omega = u \}.
\end{align*}
We remark that there are other definitions of these Sobolev spaces, for example with help of weak derivatives.
These yield norm-equivalent spaces under suitable conditions on $\Omega$, such as Lipschitz boundary, 
see e.g.\ \cite{triebel2002function, rychkov1999restrictions}.
We recall the well known Sobolev embedding theorems, 
which states for $\tau > d/2$ that $H^\tau(\Omega) \hookrightarrow \mathcal{C}(\Omega)$,
i.e.\ for any equivalence class within $H^\tau(\Omega)$ there is a continuous representer.

The following result allows to estimate intermediate Sobolev norms:

\begin{prop}
\label{prop:hoelder_interpol}
Let $\Omega \subset \R^d$ be a bounded Lipschitz domain, $\tau > 0$ and $\theta \in (0, 1)$.
Then there exists a constant $C > 0$ independent of $f$, such that
\begin{align}
\label{eq:sobolev_interpol_inequality}
\Vert f \Vert_{H^{\theta \tau}(\Omega)} \leq C \cdot \Vert f \Vert_{L^2(\Omega)}^{1-\theta} \cdot \Vert f \Vert_{H^\tau(\Omega)}^\theta \qquad \forall f \in H^\tau(\Omega).
\end{align}
\end{prop}

\begin{proof}
This is the Gagliardo–Nirenberg inequality, a well known result within the theory of interpolation of Sobolev spaces, 
see e.g.\ \cite[Theorem 1]{brezis2018gagliardo}.
\end{proof}

\subsection{Kernels and kernel interpolation}
\label{subsec:kernel_approx}

Additional to the notation introduced in \Cref{sec:introduction}, 
we recall further standard results on kernel interpolation, see mostly \cite{wendland2005scattered}:

An important class of kernels is given by \textit{translational invariant kernels}, 
for which there exists a function $\Phi: \R^d \rightarrow \R$ such that it holds
\begin{align*}
k(x, z) = \Phi(x - z) ~ \text{for all} ~ x, z \in \R^d.
\end{align*}
A particular instance of this type of kernels is given by radial basis function kernels $k$ (RBF kernels).
These can be represented with help of a \textit{radial basis function} $\varphi: \R_{\geq 0} \rightarrow \R$ as
\begin{align*}
k(x, z) = \varphi(\Vert x - z \Vert_2).
\end{align*}
Such a radial basis function kernel $k$ is again a translational invariant kernel with $\Phi(x) = \varphi(\Vert x \Vert_2)$.
These kernels are popular in applications due to their easy implementation independent of the dimension $d$ and region $\Omega$.

Translational invariant kernels and in particular RBF kernels can be characterized in terms of the Fourier transform $\hat{\Phi}$.
A frequent assumption is that it holds
\begin{align}
\label{eq:fourier_decay}
c_\Phi (1+\Vert \omega \Vert_2^2)^{-\tau} \leq \hat{\Phi}(\omega) \leq C_\Phi (1 + \Vert \omega \Vert_2^2)^{-\tau}
\end{align}
for some constants $c_\Phi, C_\Phi > 0$.
In view of the definiton of Sobolev spaces in \Cref{subsec:sobolev_spaces}, 
it follows that the RKHS $\ns$ of such kernels is norm-equivalent to the Sobolev space $H^\tau(\Omega)$ for Lipschitz regions $\Omega$, i.e.\ $\ns \asymp H^\tau(\Omega)$.
This motivates the notion of a \textit{kernel of finite smoothness $\tau > d/2$}.
Our main result \Cref{th:inverse_statement_generalized} will be formulated for such kernels of finite smoothness $\tau > d/2$, 
thus including in particular RBF kernels satisfying the decay from Eq.~\eqref{eq:fourier_decay}.

We continue by providing a detailed statement of the error estimate stated in Eq.~\eqref{eq:direct_statement},
which allows to bound the interpolation error for functions that are possibly outside the RKHS $\ns \asymp H^\tau(\Omega)$
with help of the fill distance $h_X$ and the uniformity constant $\rho_X$:

\begin{theorem}
\label{th:error_estimate}
Consider a RBF kernel $k$ that satisfies Eq.\ \eqref{eq:fourier_decay} for some $\tau > d/2$ and a compact Lipschitz region $\Omega \subset \R^d$.
For some $\beta$ with $d/2 < \beta \leq \tau$ let $f \in H^\beta(\Omega) \supseteq H^\tau(\Omega)$.
Then it holds
\begin{align*}
\Vert f - s_{f,X} \Vert_{L^2(\Omega)} \leq C h_{X}^\beta \rho_{X}^{\tau-\beta} \Vert f \Vert_{H^\beta(\Omega)}.
\end{align*}
\end{theorem}

\begin{proof}
This is a special case of \cite[Theorem 4.2]{narcowich2006sobolev} using $\mu = 0$ (with notation adopted to the current paper),
while the improvement $\beta > d/2$ (instead of $\lfloor \beta \rfloor > d/2$) is due to \cite[Theorem 4.1]{arcangeli2007extension}.
\end{proof}

For the proof of our corresponding inverse statement, we require stability statements for kernel interpolation:

\begin{theorem}
\label{th:stability_statement}
Let $k$ be a continuous kernel such that $\ns \asymp H^\tau(\Omega)$ for some Lipschitz region $\Omega \subset \R^d$.
Then there exists a constant $c_0 > 0$ such that for any set of pairwise distinct points $X \subset \Omega$ the kernel matrix $A_X$ satisfies
\begin{align*}
\lambda_{\min}(A_{X}) \geq c_0 q_{X}^{2\tau - d}
\end{align*}
and thus for the matrix norm $\Vert \cdot \Vert_{2, 2}$ it holds
\begin{align*}
\Vert A_X^{-1} \Vert_{2,2} \leq c_0^{-1} q_{X}^{d - 2\tau}.
\end{align*}
\end{theorem}

\begin{proof}
For the case of RBF kernels satisfying Eq.~\eqref{eq:fourier_decay}, which means $\ns \asymp H^\tau(\Omega)$, this is a well known results, see e.g.\ \cite{schaback1995error} or \cite[Chapter 12.2]{wendland2005scattered}.
The general case (where the kernel is not necessarily an RBF kernel) then holds due to the norm-equivalence $\ns \asymp H^\tau(\Omega)$, see e.g.~\cite[Section F.2]{haas2023mind}.
\end{proof}

Additionally to the ``continuous" $L^2(\Omega)$ norms we consider the \textit{discrete} $L^2(X)$ norm for some $X \subset \Omega$, i.e.\ for $u \in \mathcal{C}(\Omega)$ %
\begin{align}
\label{eq:discrete_L2_norm}
\Vert u \Vert_{L^2(X)} := \left( \frac{1}{|X|} \sum_{x \in X} |u(x)|^2 \right)^{1/2}.
\end{align}

The following additional proposition allows to bound the ${\Vert \cdot \Vert_{\ns}}$ norm of an kernel interpolant via a matrix norm and a discrete $L^2(X)$ norm.
The proof is based on kernel interpolation properties and we include it here for convenience:

\begin{prop}
\label{prop:upper_bound_ns_norm_via_L2}
Let $X := \{x_1, ..., x_n\} \subset \Omega$ be pairwise distinct. 
The norm $\Vert s_X \Vert_{\ns}$ of $s_X := \sum_{j=1}^n \alpha_j k(\cdot, x_j)$ can be estimated via the discrete $L^2(X)$ norm as
\begin{align*}
\Vert s_X \Vert_{\ns}^2 &\leq \Vert A_X^{-1} \Vert_{2,2} \cdot |X| \cdot \Vert s_X \Vert_{L^2(X)}^2.
\end{align*}
\end{prop}

\begin{proof}
We can calculate using $\alpha = (\alpha_1, ..., \alpha_n)^\top \in \R^n$:
\begin{align*}
\Vert s_X \Vert_{\ns}^2 &= \alpha^\top A_{X} \alpha = \alpha^\top A_{X} A_{X}^{-1} A_{X} \alpha 
\leq \Vert A_X^{-1} \Vert_{2,2} \cdot \Vert A_X \alpha \Vert_{2}^2 \\
& = \Vert A_X^{-1} \Vert_{2,2} \cdot |X| \cdot \Vert \sum_{j=1}^n \alpha_j k(\cdot, x_j) \Vert_{L^2(X)}^2 = \Vert A_X^{-1} \Vert_{2,2} \cdot |X| \cdot \Vert s_X \Vert_{L^2(X)}^2
\end{align*}
\end{proof}

\subsection{Related work} \label{subsec:related_work}

Direct and inverse statements in approximation theory are often associated with Jackson and Bernstein theorems.
The direct statements provide rates of approximation under certain assumptions of the element to be approximated,
while the inverse statements aim at deriving properties of that element based on the rate of approximation.
For classical approximation methods like polynomial, Fourier or spline approximation,
both direct and inverse statements are well researched \cite{butzer1969fundamental, devore1973saturation}.
For kernel interpolation, direct statements are also well researched \cite{wendland2005approximate, narcowich2006sobolev}, 
though the theory for corresponding inverse statements is less complete:

The standard way to prove inverse statements %
is to make use of Bernstein inequalities, 
which usually bound strong norms with help of a weaker norm. %
A typical Bernstein inequality for kernel-based interpolation on a region $\Omega \subseteq \R^d$ for a translational invariant kernel $k(x, z) = \Phi(x-z)$ satisfying Eq.~\eqref{eq:fourier_decay} and a set of pairwise distinct points $X_n \subset \Omega \subset \R^d$ reads e.g.\
\begin{align}
\label{eq:bernstein_not_available}
\left \Vert \sum_{j=1}^n \alpha_j \Phi(\cdot - x_j) \right \Vert_{H^\beta(\Omega)} \leq C q_{X_n}^{-\beta} \cdot \left \Vert \sum_{j=1}^n \alpha_j  \Phi(\cdot - x_j) \right \Vert_{L^2(\Omega)}
\end{align}
for some or all $\beta \in [0, \tau]$.
Such a Bernstein inequality Eq.~\eqref{eq:bernstein_not_available} would be helpful to prove our main result \Cref{th:inverse_statement_generalized},
however it is \textit{not} available in the literature in this general form to the best of the author's knowledge:

The statement \cite[Theorem 5.1]{narcowich2006sobolev} formulates and proves a Bernstein inequality,
however it considers $\Omega = \R^d$ instead of compact regions $\Omega \subset \R^d$ and thus cannot be used to show inverse statements on compact regions $\Omega \subset \R^d$.
In \cite{ward2012lp}, these Bernstein inequalities are extended to $L^p(\R^d)$ norms, still working on $\R^d$.
The statements in \cite[Section 6]{rieger2008sampling} generalize the $\R^d$ case to star-shaped domains $\Omega \subset \R^d$, 
however only for $\beta = \tau$.
For this, they require restrictive assumptions on the boundary and either restrict the point distribution $X \subset \Omega$  or require $\sum_{j=1}^n \alpha_j \Phi(\cdot - x_j)$ to be the interpolant of a function $f \in H^\tau(\Omega)$
(see \cite[Theorem 6.4.1]{rieger2008sampling} and \cite[Corollary 6.5.3]{rieger2008sampling}).
The result \cite[Lemma 3.2]{cheung2018convergence} states another inverse inequality,
however the assumptions on the involved norms are not appropriate for our purpose of deriving inverse statements.
The paper \cite{hangelbroek2018inverse} focusses on the case of regions $\Omega \subset \R^d$, noting that these are substantially more difficult due to the occuring boundary.
By considering localized Lagrange functions, which are obtained by augmenting the set of interpolation points with centers outside the region, 
inverse statements for such localized Lagrange functions can be derived.
Subsequently, \cite{hangelbroek2018direct} extends these techniques and results to manifolds.

The spherical case $\Omega = \mathbb{S}^{d-1} \subset \R^d$ is better researched, 
as no boundary occurs and additional tools like spherical basis functions are available.
In fact, \cite[Theorem 6.2]{narcowich2007direct} is the corresponding statement to our main result \Cref{th:inverse_statement_generalized} for the spherical case.
Also \cite{mhaskar2010bernstein} deals with this particular case of the sphere and generalizes the Bernstein inequalities to the $L^p(\mathbb{S}^{d-1})$ case,
and \cite[Lemma 9.2]{mirzaei2018petrov} generalizes these statements further.

For the more challenging general case of $\Omega \subset \R^d$ being a Lipschitz region,
the only available result so far seems to be \cite[Theorem 6.1]{schaback2002inverse}.
It states an inverse statement based on a decay assumption on $\Vert f - s_{f, X} \Vert_{L^\infty(\Omega)}$,
which can directly be applied pointwise.
However this result is not sharp, as it does not give the correct smoothness: In view of corresponding direct statements, it leaves a ``gap of $d/2$''.

All in all, sharp inverse statements for approximation by kernel interpolation are not yet available, 
and neither are the corresponding standard tools to do so.
The main result \Cref{th:inverse_statement_generalized} proven in \Cref{sec:inverse_statement} provides such sharp inverse statements by leveraging the technical tools developed in \Cref{sec:utility_statements}.

\section{Utility statements: From continuous $L^2(\Omega)$ to discrete $L^2(X)$ estimates}
\label{sec:utility_statements}

In this section we show novel results on how to obtain estimates on the discrete norm $\Vert g \Vert_{L^2(Y_1)}$ (for a continuous bounded functions $g \in \mathcal{C}(\Omega)$ and some well distributed set $Y_1 \subset \Omega$) based on estimates on the norm $\Vert g \Vert_{L^2(\Omega)}$.
The basic connection between these two norms is given by the fact, that the integral occuring within the $\Vert \cdot \Vert_{L^2(\Omega)}$ norm can be approximated by a quadrature given sufficiently many well distributed points.

In \Cref{th:utility_theorem_new} this is formalized, 
and then \Cref{prop:disc_L2_from_cont_L2_new} applies \Cref{th:utility_theorem_new} iteratively to even convert convergence rates on $\Vert f - s_{f, X_n} \Vert_{L^2(\Omega)}$ for sequences of points $(X_n)_{n \in \N} \subset \Omega$ into convergence rates on $\Vert f - s_{f, X_n} \Vert_{L^2(X_{n+1})}$.

The proofs are quite constructive and leverage the points obtained from the geometric greedy algorithm, which was recalled in \Cref{subsec:geometric_properties} and also resembles ideas from related geometric constructions as in
\cite{duchon1978erreur, rieger2008sampling, rieger2010sampling}.
As these theorems are only used as a technical tool, the proofs are deferred to \Cref{sec:proofs}.

\subsection{Continuous $L^2(\Omega)$ estimates yield discrete $L^2(X)$ estimates}

The following \Cref{th:utility_theorem_new} shows that estimates on $\Vert g \Vert_{L^2(\Omega)}$ for $g \in \mathcal{C}(\Omega)$ imply estimates on $\Vert g \Vert_{L^2(Y_1)}$ for some well distributed set of points $Y_1 \subset \Omega$.
For reasons that will become clear later on when applying these results in \Cref{sec:inverse_statement}, we aim for sets $Y_1 \subset \Omega$ which are not ``too close'' to a given reference set $Y_0 \subset \Omega$.
The reason for this is that also the set $Y_0 \cup Y_1 \subset \Omega$ needs to be well distributed (in the sense of having a sufficiently large separation distance $q_{Y_0 \cup Y_1}$).

\begin{theorem}
\label{th:utility_theorem_new}
Consider a bounded region $\Omega \subset \R^d$ which satisfies an interior cone condition with angle $\alpha \in (0, \pi/2)$ and radius $r > 0$. \\
For any bounded function $g \in \mathcal{C}(\Omega)$, any set finite reference set $Y_0 \subset \Omega$ (of at least two points) and any $q \leq \min \left( \frac{2^{-1/d} c_\Omega}{7 C_\Omega} \cdot q_{Y_0}, \frac{2}{5}r \right)$, 
there exists a finite set of points $Y_1 \subset \Omega$ with
\begin{align*}
\frac{1}{3} q \leq q_{Y_1} \leq h_{Y_1} \leq \frac{22}{3} \cdot q,
\end{align*}
and
\begin{align*}
\frac{1}{6}q \leq q_{Y_0 \cup Y_1} 
\end{align*}
such that it holds
\begin{align}
\label{eq:utility_theorem_new}
\Vert g \Vert_{L^2(Y_1)} &\leq \sqrt{\tilde{C}_{d,\alpha}} \cdot \Vert g \Vert_{L^2(\Omega)},
\end{align}
with $\tilde{C}_{d,\alpha} := 4 \cdot  \frac{16^d C_d C_\Omega^d}{c_\Omega^{2d} C_{d, \alpha}^2}$ using $C_d, C_{d, \alpha}$ from Eq.~\eqref{eq:constants_ball_cone}.
\end{theorem}

The proof can be found in \Cref{sec:proofs} and an exemplary visualization of the point sets $Y_0$ and $Y_1$ is provided in the bottom right plot of \Cref{fig:vis_proof}.
We have two comments on \Cref{th:utility_theorem_new}:

First, using the definition of the discrete norm $\Vert \cdot \Vert_{L^2(Y_1)}$ from Eq.~\eqref{eq:discrete_L2_norm},
we can rewrite Eq.~\eqref{eq:utility_theorem_new} as
\begin{align*}
\frac{1}{|Y_1|} \sum_{y \in Y_1} |g(y)|^2 &\leq  \tilde{C}_{d,\alpha} \cdot \Vert g \Vert_{L^2(\Omega)}^2,
\end{align*}
which is related to the upper bound within Marcinkiewicz-type discretization theorems, 
see e.g.\ \cite[Eq.\ (1.1)]{temlyakov2018marcinkiewicz}.
For arbitrary functions $g \in \mathcal{C}(\Omega)$, there does not seem to exist a suitable Marcinkiewicz-type discretization theorem.
However as we are interested in only a single given function $g \in \mathcal{C}(\Omega)$ instead of a whole subspace of functions,
\Cref{th:utility_theorem_new} still allows to derive such an estimate.
In fact, the set $Y_1$ may depend on the function $g$.

Second, as it can be easily seen from the proof of \Cref{th:utility_theorem_new}, 
analogous inequalities using $L^p$ norms with $1 \leq p < \infty$ in Eq.~\eqref{eq:utility_theorem_new} can be derived, 
as the proof was based on a discretization of the occuring integrals.
In order to keep it more simple, we formulated the theorem only in terms of $p=2$, which is the case which will be leveraged in the following.

\subsection{Continuous $L^2(\Omega)$ convergence rates yield discrete $L^2(X)$ convergence rates}

In the following,
we consider the interpolation of continuous functions via kernel interpolation and their corresponding convergence rates.
The corresponding convergence rates are usually given in the ``continuous'' $\Vert \cdot \Vert_{L^2(\Omega)}$ norm for well distributed points, see \Cref{th:error_estimate}.

We aim at using \Cref{th:utility_theorem_new} iteratively for $g := f - s_{f, X_n}$ in order to obtain convergence rates in a discrete $\Vert f - s_{f, X_n} \Vert_{L^2(X_{n+1})}$ norm 
for some sequence $(X_n)_{n \in \N} \subset \Omega$ based on convergence rates in the continuous $\Vert f - s_{f, X_n} \Vert_{L^2(\Omega)}$ norm.

\begin{prop}
\label{prop:disc_L2_from_cont_L2_new}
Consider a bounded region $\Omega \subset \R^d$ which satisfies an interior cone condition with angle $\alpha \in (0, \pi/2)$ and radius $r>0$. \\
Consider a function $f \in \mathcal{C}(\Omega)$ and a continuous kernel $k$ such that for some $c_f, \mu, h_0 > 0$ it holds
\begin{align*}
\Vert f - s_{f, X} \Vert_{L^2(\Omega)} \leq c_f h_X^{\mu}
\end{align*}
for all $X \subset \Omega$ with $h_X \leq h_0$ and $\rho_X \leq 44$ (quasi-uniform points). \\
Then there exist constants $a \in (0, \frac{3}{154})$ and $c_0 > 0$ and a sequence of nested sets of points $X_0 \subset X_1 \subset ... \subset X_m \subset .. \subset \Omega$ with 
\begin{align*}
\frac{1}{6} c_0 a^{m+1} \leq q_{X_{m+1}} \leq h_{X_{m+1}} \leq \frac{22}{3} c_0 a^{m+1}
\end{align*}
such that
\begin{align*}
\Vert f - s_{f, X_m} \Vert_{L^2(X_{m+1})} \leq c_f \sqrt{\tilde{C}_{d,\alpha}} h_{X_m}^{\mu}
\end{align*}
for all $m > 0$, using $\tilde{C}_{d,\alpha}$ defined in \Cref{th:utility_theorem_new}.
\end{prop}

The proof can be found in \Cref{sec:proofs}.

\section{Inverse statement: Continuous $L^2(\Omega)$ convergence rates yield smoothness} \label{sec:inverse_statement}

In the following we state and prove our main result \Cref{th:inverse_statement_generalized},
which is an inverse statement: 
Given a convergence rate on $\Vert f - s_{f, X} \Vert_{L^2(\Omega)}$ in the fill distance $h_X$ for the approximation of a function $f \in \mathcal{C}(\Omega)$,
we derive a corresponding Sobolev smoothness of the function $f$.

\mainresult*

The idea of the proof is to consider a sequence of interpolants on a sequence of nested and increasingly denser point sets.
Then one shows that this sequence is a Cauchy sequence in a Sobolev space of some smoothness,
and by completeness of the Sobolev spaces one then obtains a limiting element.

\begin{proof}[Proof of \Cref{th:inverse_statement_generalized}]
We start by using \Cref{prop:disc_L2_from_cont_L2_new} for $\mu := \beta$:
Like this we obtain constants $a \in (0, \frac{3}{154})$ and $c_0 > 0$ and a sequence of nested sets of points $(X_m)_{m \geq 0}$ such that
\begin{align}
\begin{aligned}
\label{eq:proof:properties_sequence}
\frac{1}{6} c_0 a^{m+1} \leq& ~ q_{X_{m+1}} \leq h_{X_{m+1}} \leq \frac{22}{3} c_0 a^{m+1}, \\
\Vert f - s_{f, X_m} \Vert_{L^2(X_{m+1})} \leq& ~ c_f \sqrt{\tilde{C}_{d,\alpha}} h_{X_m}^{\beta}
\end{aligned}
\end{align}
Observe that these sets of points $X_m$ have a uniformity constant $\rho_{X_m}$ bounded as $\rho_{X_m} \equiv \frac{h_{X_m, \Omega}}{q_{X_m}} \leq 44$.

We consider $\beta' < \beta$ for $\beta \in (0, \tau]$ respective $\beta'=\tau$ for $\beta > \tau$ and 
continue by applying \Cref{prop:hoelder_interpol} to $g := s_{f, X_{n+1}} - s_{f, X_n} \in \ns \asymp H^\tau(\Omega)$ using $\theta = \beta'/\tau$ for $\beta \in (0, \tau]$ or $\theta = 1$ for $\beta > \tau$:
\begin{align}
\begin{aligned}
\label{eq:application_hoelder}
&\Vert s_{f, X_{m+1}} - s_{f, X_{m}} \Vert_{H^{\beta'}(\Omega)}^2 \\
\leq& ~ C \cdot \Vert s_{f, X_{m+1}} - s_{f, X_{m}} \Vert_{L^2(\Omega)}^{2 - 2\theta} \cdot \Vert s_{f, X_{m+1}} - s_{f, X_{m}} \Vert_{H^\tau(\Omega)}^{2\theta} \\
\leq& ~ C' \cdot\Vert s_{f, X_{m+1}} - s_{f, X_{m}} \Vert_{L^2(\Omega)}^{2 - 2\theta} \cdot \Vert s_{f, X_{m+1}} - s_{f, X_{m}} \Vert_{\ns}^{2\theta} \\
\leq& ~ C' \cdot \Vert s_{f, X_{m+1}} - s_{f, X_{m}} \Vert_{L^2(\Omega)}^{2 - 2\theta} \\ 
&\qquad \cdot \Vert A_{X_{m+1}}^{-1} \Vert_{2,2}^{\theta} \cdot |X_{m+1}|^{\theta} \cdot \Vert s_{f, X_{m+1}} - s_{f, X_{m}} \Vert_{L^2(X_{m+1})}^{2\theta},
\end{aligned}
\end{align}
where we estimated the $\Vert s_{f, X_{m+1}} - s_{f, X_{m}} \Vert_{\ns}$ quantity by using \Cref{prop:upper_bound_ns_norm_via_L2}.
For these four factors we have the following estimates:
\begin{itemize}
\item Adding $0 = f - f$, leveraging the assumed bound on $\Vert f - s_{f, X} \Vert_{L^2(\Omega)}$ and then estimating the fill distances as in Eq.\ \eqref{eq:proof:properties_sequence} gives
\begin{align*}
&\Vert s_{f, X_{m+1}} - s_{f, X_{m}} \Vert_{L^2(\Omega)} \\
\leq \, &\Vert f - s_{f, X_{m+1}} \Vert_{L^2(\Omega)} + \Vert f - s_{f, X_{m}} \Vert_{L^2(\Omega)} \\
\leq \, &2c_f h_{X_m}^{\beta} = 2c_f \left( \frac{22c_0}{3} \right)^\mu a^{\beta m}.
\end{align*}
\item For the norm of the inverse matrix we use \Cref{th:stability_statement} and subsequently the bound from Eq.\ \eqref{eq:proof:properties_sequence} on the separation distance to obtain
\begin{align*}
\Vert A_{X_{m+1}}^{-1} \Vert_{2,2} \leq& ~ c_0^{-1} q_{X_{m+1}}^{d-2\tau} \leq c_0^{-1} \left(\frac{1}{6} c_0 a^{m+1} \right)^{d-2\tau} \\
=& ~ \frac{6^{2\tau-d}}{c_0^{2\tau - d + 1} a^{2\tau-d}} a^{-m (2\tau - d)}.
\end{align*}
\item The number of points $|X_{m+1}|$ can be upper estimated via Eq.\ \eqref{eq:geometric_bounds} as 
\begin{align*}
|X_{m+1}| \leq C_\Omega^d q_{X_{m+1}}^{-d} \leq \frac{6^d C_\Omega^d}{c_0^d} a^{-d(m+1)} = \frac{6^d C_\Omega^d}{c_0^d a^d} a^{-dm}.
\end{align*}
\item Using $s_{f, X_{m+1}}|_{X_{m+1}} = f|_{X_{m+1}}$, 
the bound on $\Vert f - s_{f, X_m} \Vert_{L^2(X_{m+1})}$ from Eq.\ \eqref{eq:proof:properties_sequence}, and subsequently estimating the fill distances via Eq.\ \eqref{eq:proof:properties_sequence} gives
\begin{align*}
&\Vert s_{f, X_{m+1}} - s_{f, X_{m}} \Vert_{L^2(X_{m+1})} = \Vert f - s_{f, X_{m}} \Vert_{L^2(X_{m+1})} \\
\leq \, &c_f \cdot \sqrt{\tilde{C}_{d, \alpha}} \cdot h_{X_m}^\beta \leq \, c_f \cdot \sqrt{\tilde{C}_{d, \alpha}} \cdot \left( \frac{22}{3} c_0 a^m \right)^\beta \\
\leq \, &c_f \cdot \sqrt{\tilde{C}_{d, \alpha}} \cdot \left( \frac{22c_0}{3} \right)^\beta a^{\beta m}.
\end{align*}
\end{itemize}

\noindent Putting everything together we obtain %
\begin{align*}
&\Vert s_{f, X_{m+1}} - s_{f, X_{m}} \Vert_{H^{\beta'}(\Omega)}^2 \\
\leq &\, C' \cdot \Vert s_{f, X_{m+1}} - s_{f, X_{m}} \Vert_{L^2(\Omega)}^{2 - 2\theta} \cdot \Vert A_{X_{m+1}}^{-1} \Vert_{2,2}^{\theta} \cdot |X_{m+1}|^{\theta} \cdot  \Vert s_{f, X_{m+1}} - s_{f, X_{m}} \Vert_{L^2(X_{m+1})}^{2\theta} \\
\leq &\,C'' \cdot a^{2\beta m (1-\theta)} \cdot a^{-m (2\tau - d) \theta} \cdot a^{-md \theta} \cdot a^{2\beta m \theta} \\
= &\,C'' \cdot a^{m(\theta \cdot [-2\beta - 2\tau + d - d + 2\beta] + 2\beta)} \\
= &\,C'' \cdot a^{2m(\beta-\beta')},
\end{align*}
with a generic constant $C'' > 0$ given by the collection of the previously occuring constants. %
As $\beta' < \beta$ we obtain $\beta - \beta' =: \delta > 0$.
Then we can show that $(s_{f, X_m})_{m \geq 1}$ is a Cauchy series, i.e. for $n > m \geq m_0$:
\begin{align}
\label{eq:cauchy_sequence_calculation}
\Vert s_{f, X_{n}} - s_{f, X_m} \Vert_{H^{\beta'}(\Omega)} &= \left \Vert \sum_{\ell=m}^{n-1} s_{f, X_{\ell+1}} - s_{f, X_{\ell}} \right \Vert_{H^{\beta'}(\Omega)} \notag \\
&\leq \sum_{\ell=m}^{n-1} \Vert s_{f, X_{\ell+1}} - s_{f, X_{\ell}} \Vert_{H^{\beta'}(\Omega)} \notag \\
&\leq \sqrt{C''} \cdot \sum_{\ell=m}^{\infty} a^{\ell\delta} 
\leq \sqrt{C''} \cdot \frac{a^{\delta m_0}}{1 - a^{\delta}} \stackrel{m_0 \rightarrow \infty}{\longrightarrow} 0.
\end{align}
As the Sobolev space $H^{\beta'}(\Omega)$ is a complete space and $(s_{f, X_m})_{m \geq m_0}$ is a Cauchy sequence in $H^{\beta'}(\Omega)$, 
there exists a unique limit element $\tilde{f} \in H^{\beta'}(\Omega)$. \\
It remains to show that $\tilde{f}$ and $f$ coincide on $\Omega$. 
For this we consider
\begin{align*}
\Vert f - \tilde{f} \Vert_{L^2(\Omega)} %
&\leq \Vert f - s_{f, X_n} \Vert_{L^2(\Omega)} + \Vert \tilde{f} - s_{f, X_n} \Vert_{L^2(\Omega)} \\
&\leq c_f h_{X_n}^{\mu} + C_{H^{\beta'}(\Omega) \hookrightarrow L^2(\Omega)} \Vert \tilde{f} - s_{f, X_n} \Vert_{H^{\beta'}(\Omega)} \stackrel{n \rightarrow \infty}{\longrightarrow} 0.
\end{align*}
Thus we have $f = \tilde{f}$ in $L^2(\Omega)$ an $f$ can be seen as a particular representative of $\tilde{f}$, 
i.e.\ we have proven that $f \in H^{\beta'}(\Omega)$ for any $\beta' < \beta$. 
\end{proof}

\noindent In the following we collect some remarks on the theorem and the proof:
\begin{itemize}
\item First we remark that the inequalities in Eq.~\eqref{eq:application_hoelder} essentially bounds a strong norm (the $\Vert \cdot \Vert_{H^{\beta'}(\Omega)}$ norm) via weaker norms 
(the $\Vert \cdot \Vert_{L^2(\Omega)}$ norm and the $\Vert \cdot \Vert_{L^2(X_m \cup X_{m+1})}$ norm, which can be again bounded via the $\Vert \cdot \Vert_{L^2(\Omega)}$ norm).
Thus this can be seen as a Bernstein inequality on general Lipschitz regions $\Omega \subset \R^d$.
However it is not a Bernstein inequality in its full generality, 
as it holds only with respect to the derived sequence $(X_m)_{m \geq m_0} \subset \Omega$ of  sets of points.
\item Second we remark that the statement of \Cref{th:inverse_statement_generalized} can be refined by additionally taking into account logarithmic decay rates, 
i.e.\ within \Cref{th:inverse_statement_generalized} we could assume a decay as
\begin{align*}
\Vert f - s_{f, X} \Vert_{L^2(\Omega)} \leq c_f h_X^{\beta} \cdot |\log(h_X)|^{-\nu}
\end{align*}
for all $X \subset \Omega$ with $h_X \leq h_0$ and $\rho_X \leq 44$ (quasi-uniformly distributed) and some $\nu \in \R$. \\
Due to $h_{X_m} \stackrel{\bigcdot}{\leq} a^m$ in the proof of \Cref{th:inverse_statement_generalized} (where $\stackrel{\bigcdot}{\leq}$ denotes inequality up to a constant), 
the decay $h_{X_m}^{\mu}$ turned into quantities like $a^{\mu m}$.
And thus the parts $|\log(h_{X_m})|^{-\nu}$ will turn into $m^{-\nu} |\log(a)|^{-\nu} \stackrel{\bigcdot}{=} m^{-\nu}$.
Therefore, when showing the Cauchy sequence property in Eq.~\eqref{eq:cauchy_sequence_calculation}, 
we would not obtain the pure geometric series $\sum_{\ell=m}^{\infty} a^{\ell\delta}$ anymore.
Instead, in particular for $\beta = \beta'$, we would obtain a harmonic sum like $\sum_{\ell=m}^{\infty} \ell^{-\tilde{\delta} \nu}$ (for some constant $\tilde{\delta} > 0$) which still converges for $\nu > 0$ large enough.
We waived to include this more detailed but more complicated analysis here in order to keep a clear outline of the exposition.
\end{itemize}

Finally we point out that in view of \Cref{th:error_estimate} the inverse statement in \Cref{th:inverse_statement_generalized} is sharp for $\Omega \subset \R^d$ with Lipschitz boundary (up to an arbitrary small $\varepsilon > 0$):

For $f \in H^\beta(\Omega)$ with $d/2 < \beta \leq \tau$, 
the direct statement \Cref{th:error_estimate} provides an $\Vert \cdot \Vert_{L^2(\Omega)}$ error decay for asymptotically uniformly distributed points $X \subset \Omega$ according to $h_{X}^\beta$.
On the other hand given an approximation rate as $h_X^\beta$, the inverse statement \Cref{th:inverse_statement_generalized} already provides $f \in H^{\beta'}(\Omega)$ for any $\beta' < \beta$.
This establishes a one-to-one correspondence between smoothness and approximation rate for kernel based interpolation using finitely smooth kernels.
\section{Discussion} \label{sec:discussion_remarks}

We start by commenting on the relation between the main result \Cref{th:inverse_statement_generalized} and the related literature discussed in \Cref{subsec:related_work}:
The standard approach for showing inverse statements is with help of Bernstein inequalities.
As these are not available in the literature for the general case of compact Lipschitz domains because of challenges due to the boundary,
we employed the utility statements of \Cref{sec:utility_statements} instead.
These statements allowed to turn the $L^2(\Omega)$ convergence rates into point based estimates, 
which enabled the construction of a sequence of interpolants that turned out to be a Cauchy sequence.
As elaborated in the previous section, the resulting estimate can also be understood as a Bernstein inequality.
While it is evident that the utility statements of \Cref{sec:utility_statements} are generalizable to other $L^p(\Omega)$ norms, $1 \leq p < \infty$,
the interpolation inequality within the key step Eq.~\eqref{eq:application_hoelder} does not transfer in a straightforward way.
Thus it remains a future task to transfer the $L^2(\Omega)$ error analysis of the inverse statement from \Cref{th:inverse_statement_generalized} into $L^\infty(\Omega)$ or other $L^p(\Omega)$ error statements that are sharp as well. 

For the case of $\Omega$ being a compact Lipschitz region, 
the only inverse statement so far was given in \cite{schaback2002inverse} and was formulated in terms of $L^\infty(\Omega)$ convergence rates.
In contrast to \Cref{th:inverse_statement_generalized}, 
it covered only the edge case of smoothness $\tau$ (instead of intermediate smoothnesses $\beta \in (0, \tau]$) and did not yield a sharp inverse statement, thus leaving a ``gap of $d/2$''.
Such a gap is no longer existent in the present approach which considers $L^2(\Omega)$ convergence rates.
In fact, our inverse statement \Cref{th:inverse_statement_generalized} is exactly \cite[Theorem 6.2]{narcowich2007direct},
however for the case of a compact Lipschitz region $\Omega$ instead of the sphere $\mathbb{S}^{d-1}$. 

\Cref{th:inverse_statement_generalized} makes use of well distributed points due to considering the fill distance $h_{X}$.
This is a frequent assumption in kernel based interpolation, 
and also useful in applications since greedy algorithms as the $P$-greedy algorithm are known to provide well distributed points \cite[Theorem 19 \& 20]{wenzel2021novel}.
It remains open whether and how inverse statements can be generalized to adaptively distributed points, 
such as greedily selected points \cite{wenzel2023analysis, santin2024optimality}.

The key ingredients for the proof of \Cref{th:inverse_statement_generalized} have been the technical statement of \Cref{prop:disc_L2_from_cont_L2_new} as well as standard RKHS analysis tools including stability estimates as in \Cref{th:stability_statement}.
These tools do not seem to rely too much on the considered domain $\Omega \subset \R^d$, such that an extension to manifolds seems possible: 
In the case of manifolds, the asymptotic of fill and separation distance for well distributed points still coincide 
-- in view of Eq.~\eqref{eq:geometric_bounds} they will however depend on the effective dimension of the manifold instead of the dimension of the ambient space.
Also weakening the assumption of a Lipschitz boundary or cone condition might be possible, 
as \cite{wenzel2024stability} showed that the $L^\infty(\Omega)$ error estimates do not depend significantly on the shape of the region $\Omega \subset \R^d$ 
-- though it remains open whether this holds as well for the $L^2(\Omega)$ error.

We continue with a remark on the sharpness of the inverse statements, which allowed to conclude the one-to-one correspondence between smoothness and approximation rates as elaborated after \Cref{th:inverse_statement_generalized}:
We considered continuous kernels satisfying a norm-equivalence $\ns \asymp H^\tau(\Omega$ for some $\tau > d/2$, 
which allowed to make use of stability statements which were originally derived for RBF kernels, see \Cref{th:stability_statement}.
Such RBF kernels are usually characterized with help of the decay of their Fourier transform, i.e.\ Eq.~\eqref{eq:fourier_decay}:
The upper bound on the Fourier decay usually yields the error estimates (via the fill distance), i.e.\ the direct statements,
while the lower bound on the Fourier decay yields stability estimates (via the separation distance) which are required for the inverse statements \cite{schaback2002inverse, schaback1995error}.
In case of a gap between the lower and upper bound, i.e.
\begin{align*}
c_\Phi (1+\Vert \omega \Vert_2^2)^{-\tau_1} \leq \hat{\Phi}(\omega) \leq C_\Phi (1 + \Vert \omega \Vert_2^2)^{-\tau_2}
\end{align*}
for $d/2 < \tau_1 < \tau_2$,
this could result in a gap between direct and inverse statements.
Such different exponents, i.e.\ $\tau_1 < \tau_2$, can happen if the Fourier transform $\hat{\Phi}$ has a staircase behavior.
We remark that such a kernel is probably rather exotic and of little practical use.

In this paper we considered kernels of finite smoothness, i.e.\ $\ns \asymp H^\tau(\Omega)$.
However there are also kernels of infinite smoothness, such as the Gaussian kernel $k(x, z) = \exp(- \Vert x - z \Vert_2^2)$.
A one-to-one correspondence as discussed after \Cref{th:inverse_statement_generalized}) might not be derivable here:
For example for the Gaussian kernel, the known error estimates give rates as $\exp(c \log(h_{X}) / h_{X})$ \cite[Theorem 11.22]{wendland2005scattered}.
However the corresponding stability estimates look like $\exp(-40.71d^2/q_X^2) q_X^{-d}$ \cite[Corollary 12.4]{wendland2005scattered}, 
i.e.\ there is a large gap even for well distributed points, i.e.\ $q_X \asymp h_{X}$.
A detailed investigation on this possible \textit{analytic gap} between direct and inverse statements seems to be an interesting future research direction.
Another possible direction for future work is to improve the case $\beta > \tau$ within \Cref{th:inverse_statement_generalized}, 
as one would naturally expect to conclude increased smoothness here as well.
Within the presented analysis this was not possible, as \Cref{prop:hoelder_interpol} is no longer applicable for $\beta > \tau$.

Note that \Cref{th:inverse_statement_generalized} and thus the one-to-one correspondence also hold for large values of the dimension $d$,
as only the asymptotic behavior of the error for $h_{X} \rightarrow 0$ is of importance.
Though the dimensionality will affect the constants involved.
However when dealing with only finitely many scattered data points in high dimensions, 
the fill distance is usually larger than one.
We point to \cite{lin2024kernel} for more details on kernel interpolation of scattered data in high dimensions.

Note that \Cref{th:inverse_statement_generalized} also works for large values of the dimension $d$,
as long as Eq.~\eqref{eq:inverse_statement_assumption} is satisfied.
The dimensionality only affects the constants involved, however not the asymptotic behavior of the error for $h_{X, \Omega} \rightarrow 0$,
which is only linked to the Sobolev smoothness .
We point to \cite{lin2024kernel} for details on kernel interpolation in high-dimensional domains.

Finally, we conclude with a numerical example related to the one-to-one correspondence,
i.e.\ \Cref{th:error_estimate} and \Cref{th:inverse_statement_generalized}:

\begin{example} %
We consider the reconstruction of continuous functions on $\Omega := [0, 1]$ using the Matérn kernel $k(x, z) = \exp(-\Vert x - z \Vert_2)$,
 which is an RBF kernel with $\Phi(r) = \exp(-r)$ and corresponding Fourier transform $\hat{\Phi}(\omega) = \sqrt{\frac{2}{\pi}} \cdot \frac{1}{\omega^2 + 1}$.
Thus, the RKHS $\ns$ is norm equivalent to the Sobolev space $H^1(\Omega)$, i.e.\ $\tau = 1$.
We consider the functions
\begin{align*}
f_\sigma: \Omega \rightarrow \R, x \mapsto x^\sigma \cdot (1-x)^\sigma
\end{align*}
for $\sigma > 0$ and it can be shown that $f_\sigma \in H^{\sigma' + 1/2}(\Omega)$ ~ $\forall \sigma' \in (0, \sigma)$ for any $\sigma \in (0, 2) \setminus \{ 1 \}$.
For the numerical experiment, we make use of the values %
$\sigma \in \{ 0.21, 0.36, 0.51, 0.81, 1.11, 1.51 \}$ 
and consider the $\Vert f_\sigma - s_{f_\sigma, X} \Vert_{L^2(\Omega)}$ error using up to $2^{13} + 1 = 8193$ equidistant interpolation points $X \subset \Omega$,
where each both boundary points are included.
The $\Vert f_\sigma - s_{f_\sigma, X} \Vert_{L^2(\Omega)}$ error was evaluated by discretizing $\Omega$ using $2^{20} + 1$ equally spaced points.

The results are visualized in \Cref{fig:matern_L2_error}:
For every value of $\sigma$, the decay of the error $\Vert f_\sigma - s_{f_\sigma, X} \Vert_{L^2(\Omega)}$ in the number of interpolation points $n$ is depicted.
The legend additionally lists the numerically computed decay rates,
and the dashed black lines indicate the theoretically expected decay.
For $\sigma \in \{ 0.21, 0.36 \}$ we have $f_\sigma \in H^{\sigma' + 1/2} \supset H^\tau(\Omega)$ ~ $\forall \sigma' \in (0, \sigma)$, 
i.e.\ exactly the situation described in \Cref{th:error_estimate} and \Cref{th:inverse_statement_generalized}: 
We obtain a decay as $n^{-0.69}$ and $n^{-0.85}$, which almost exactly match the expected decay of $n^{-0.71}$ respectively $n^{-0.86}$.
For $\sigma \in \{ 0.51, 0.81, 1.11, 1.51 \}$ we have $\sigma + 1/2 > 1$ and thus $f_\sigma \in H^1(\Omega)$, 
i.e.\ additional smoothness and apparently an increased convergence rate. %
The dashed black lines indicate likewise a decay as $n^{-1.01}$, $n^{-1.31}$, $n^{-1.61}$ and $n^{-2.01}$, i.e.\ scaling according to the smoothness of $f_\sigma$.
This seems to be an instance of superconvergence, see e.g.\ \cite{schaback2018superconvergence}.

\begin{figure}[hbt!]
\centering
\setlength\fwidth{.65\textwidth}
\begin{tikzpicture}

\definecolor{crimson2143940}{RGB}{214,39,40}
\definecolor{darkgray176}{RGB}{176,176,176}
\definecolor{darkorange25512714}{RGB}{255,127,14}
\definecolor{forestgreen4416044}{RGB}{44,160,44}
\definecolor{lightgray204}{RGB}{204,204,204}
\definecolor{mediumpurple148103189}{RGB}{148,103,189}
\definecolor{sienna1408675}{RGB}{140,86,75}
\definecolor{steelblue31119180}{RGB}{31,119,180}

\begin{axis}[
legend cell align={left},
legend style={
  fill opacity=0.8,
  draw opacity=1,
  text opacity=1,
  at={(0.03,0.03)},
  anchor=south west,
  draw=lightgray204
},
log basis x={10},
log basis y={10},
tick align=outside,
tick pos=left,
x grid style={darkgray176},
xmin=3.45339580257623, xmax=11862.2371549303,
xmode=log,
xtick style={color=black},
xtick={0.1,1,10,100,1000,10000,100000,1000000},
xticklabels={
  \(\displaystyle {10^{-1}}\),
  \(\displaystyle {10^{0}}\),
  \(\displaystyle {10^{1}}\),
  \(\displaystyle {10^{2}}\),
  \(\displaystyle {10^{3}}\),
  \(\displaystyle {10^{4}}\),
  \(\displaystyle {10^{5}}\),
  \(\displaystyle {10^{6}}\)
},
y grid style={darkgray176},
ymin=5.78622928893346e-09, ymax=0.280003478947575,
ymode=log,
ytick style={color=black},
ytick={1e-10,1e-09,1e-08,1e-07,1e-06,1e-05,0.0001,0.001,0.01,0.1,1,10},
yticklabels={
  \(\displaystyle {10^{-10}}\),
  \(\displaystyle {10^{-9}}\),
  \(\displaystyle {10^{-8}}\),
  \(\displaystyle {10^{-7}}\),
  \(\displaystyle {10^{-6}}\),
  \(\displaystyle {10^{-5}}\),
  \(\displaystyle {10^{-4}}\),
  \(\displaystyle {10^{-3}}\),
  \(\displaystyle {10^{-2}}\),
  \(\displaystyle {10^{-1}}\),
  \(\displaystyle {10^{0}}\),
  \(\displaystyle {10^{1}}\)
}
]
\addplot [semithick, steelblue31119180, mark=x, mark size=5, mark options={solid}, only marks]
table {%
5 0.0466967065630209
9 0.0478759498328178
17 0.0366928179821262
33 0.024860053254987
65 0.0159639205368953
129 0.00999936475371457
257 0.00619026530859231
513 0.00381240517195436
1025 0.00234479220433814
2049 0.00144477068896539
4097 0.000895486726452085
8193 0.000562119428182609
};
\addlegendentry{$\sigma$ = 0.21, $n^{-0.69}$}
\addplot [semithick, darkorange25512714, mark=x, mark size=5, mark options={solid}, only marks]
table {%
5 0.0697358374924028
9 0.0569537027628708
17 0.0371838280793163
33 0.0221739431899068
65 0.0126942524694815
129 0.00712789011388194
257 0.00396513965147717
513 0.00219609361367931
1025 0.00121422393322513
2049 0.00067141472289871
4097 0.000372011851195291
8193 0.000207155566541475
};
\addlegendentry{$\sigma$ = 0.36, $n^{-0.85}$}
\addplot [semithick, forestgreen4416044, mark=x, mark size=5, mark options={solid}, only marks]
table {%
5 0.075541297598699
9 0.0532783496987298
17 0.0308363709378836
33 0.016457413027936
65 0.0084639073698255
129 0.00427644264895892
257 0.00214216475070379
513 0.00106863856079846
1025 0.00053215515675901
2049 0.000264915239379014
4097 0.000132009547594661
8193 6.5968050457666e-05
};
\addlegendentry{$\sigma$ = 0.51, $n^{-1.00}$}
\addplot [semithick, crimson2143940, mark=x, mark size=5, mark options={solid}, only marks]
table {%
5 0.0648008188362111
9 0.0366075789899073
17 0.0171349829740573
33 0.00741653916940206
65 0.00309616456120359
129 0.00127025922857185
257 0.000516734408817335
513 0.000209337643578268
1025 8.46457933686391e-05
2049 3.42056082368218e-05
4097 1.38275281985738e-05
8193 5.59816640782513e-06
};
\addlegendentry{$\sigma$ = 0.81, $n^{-1.31}$}
\addplot [semithick, mediumpurple148103189, mark=x, mark size=5, mark options={solid}, only marks]
table {%
5 0.047421540343049
9 0.0221105283577681
17 0.0084724650073319
33 0.00299030003759234
65 0.00101594650587007
129 0.00033887945146583
257 0.00011202360055803
513 3.68690778296606e-05
1025 1.2109266953291e-05
2049 3.97406564258091e-06
4097 1.30437114496383e-06
8193 4.28573685975758e-07
};
\addlegendentry{$\sigma$ = 1.11, $n^{-1.61}$}
\addplot [semithick, sienna1408675, mark=x, mark size=5, mark options={solid}, only marks]
table {%
5 0.0283674790211751
9 0.0104127842643092
17 0.00308245144902635
33 0.00083354641630121
65 0.000216154221300006
129 5.49354201340526e-05
257 1.38254619897806e-05
513 3.46284208466582e-06
1025 8.6537885824424e-07
2049 2.16063513561059e-07
4097 5.3941825847535e-08
8193 1.34765354532366e-08
};
\addlegendentry{$\sigma$ = 1.51, $n^{-2.00}$}
\addplot [semithick, black, dashed, forget plot]
table {%
5 0.111634733840691
9 0.073545458981321
17 0.046821929873124
33 0.0292363496419669
65 0.0180675623091796
129 0.0111057872618015
257 0.0068079369782053
513 0.00416758738826269
1025 0.00254949462044316
2049 0.00155909742364156
4097 0.000953272818856703
8193 0.000582805372242441
};
\addplot [semithick, black, dashed, forget plot]
table {%
5 0.125272516182551
9 0.0755651321550609
17 0.0437304485456777
33 0.0247199743646695
65 0.0137995370750179
129 0.00765354876153395
257 0.00423084892923338
513 0.0023349042213256
1025 0.00128750071857208
2049 0.000709649531044706
4097 0.000391065294134118
8193 0.000215481040096465
};
\addplot [semithick, black, dashed, forget plot]
table {%
5 0.118084133203615
9 0.0652178257152781
17 0.0343081926259354
33 0.017557075273129
65 0.00885337296697497
129 0.00443052960605404
257 0.00220860878768589
513 0.00109883557571443
1025 0.000546160322562205
2049 0.00027132753693993
4097 0.000134759883478144
8193 6.69227634405898e-05
};
\addplot [semithick, black, dashed, forget plot]
table {%
5 0.0971492117820942
9 0.0449813307451723
17 0.0195524822357925
33 0.00820043188796158
65 0.00337422073423558
129 0.00137473017837649
257 0.000557285310245906
513 0.000225338935663181
1025 9.10000693577902e-05
2049 3.67256948581593e-05
4097 1.48169758903464e-05
8193 5.97695128479253e-06
};
\addplot [semithick, black, dashed, forget plot]
table {%
5 0.0636908754939917
9 0.0247222779882753
17 0.00887965642656872
33 0.00305218941941043
65 0.00102477476292188
129 0.000339914344334921
257 0.000112053802093367
513 3.68239157183055e-05
1025 1.20824081527805e-05
2049 3.96128758364237e-06
4097 1.29822128026726e-06
8193 4.25378728130128e-07
};
\addplot [semithick, black, dashed, forget plot]
table {%
5 0.0373933088478114
9 0.011473506375836
17 0.00319537088182732
33 0.000842384924720835
65 0.00021565908509298
129 5.43798853456246e-05
257 1.36068634779611e-05
513 3.39146782627912e-06
1025 8.43662286884708e-07
2049 2.09664193991324e-07
4097 5.20795253861513e-08
8193 1.2933118773864e-08
};
\end{axis}

\end{tikzpicture}	
\caption{Visualization of the $L^2(\Omega)$ error decay in the ($y$-axis) for interpolation with the Matérn kernel $k(x, z) = \exp(-\Vert x - z \Vert)$ on $\Omega = [0,1]$ in the number of equidistant interpolation points ($x$-axis) for different $\sigma$ parameters of the functions $x^\sigma \cdot (1-x)^\sigma$.
This class of functions scales linearly in the Sobolev spaces $H^\tau(\Omega)$: 
It holds $x^\sigma \cdot (1-x)^\sigma \in H^{\sigma + 1/2}(\Omega)$ ~ $\forall \sigma' \in (0, \sigma)$.
The dashed black lines indicate the convergence rates $n^{-0.71}, n^{-0.86}, n^{-1.01}$, $n^{-1.31}$, $n^{-1.61}$ and $n^{-2.01}$,
while the legend lists numerically computed convergence rates.}
\label{fig:matern_L2_error}
\end{figure}

\end{example}

\section{Conclusion} \label{sec:conclusion}

The current paper proved sharp inverse statements for kernel interpolation using finitely smooth kernels for which the RKHS $\ns$ is norm-equivalent to a Sobolev space $H^\tau(\Omega)$, $\tau > d/2$.
This result includes in particular popular RBF kernels like the class of Matérn or Wendland kernels.
As a techincal tool to do so, discrete $L^2(X)$ error estimates based of continuous $L^2(\Omega)$ error estimates were derived.

As a result of the sharpness of the inverse statement, 
a one-to-one correspondence between smoothness (of the function to approximate) and the corresponding approximation rate was concluded. %
A numerical experiment illustrated the result.
The derived techniques and results are expected to yield further insights, e.g.\ for kernel quadrature, PDEs approximation with kernels or in statistical learning theory. \\
For future research, we aim at several directions:
First, investigations of other kernels such as analytic kernels as the Gaussian and other $L^p(\Omega)$, $1 \leq p \leq \infty$ norms seems meaningful.
Second, the results should be extended to the superconvergence case, i.e.\ the case where the considered function is smoother than the RKHS. \\

\textbf{Acknowledgements:}
The work was mainly conducted while the author was affiliated with the University of Stuttgart.
The author thanks Prof.\ Bernard Haasdonk and Prof.\ Gabriele Santin for support, discussions and feedback
and acknowledges funding of the project by the Deutsche Forschungsgemeinschaft (DFG, German Research Foundation) under Germany's Excellence Strategy - EXC 2075 - 390740016 and funding by the BMBF under contract 05M20VSA as well as support from the Studienstiftung des deutschen Volkes (German national Academic Foundation).
Part of the work was inspired and motivated by a research stay of the author with Prof.\ Lorenzo Rosasco at MaLGa, DIBRIS, Universit`a di Genova.

\IfFileExists{/home/wenzel/references.bib}{
\bibliography{/home/wenzel/references}			%
}{
\bibliography{/home/math/wenzel/references}			%

\begin{thebibliography}{10}

\bibitem{agranovich2015sobolev}
M.~S. Agranovich.
\newblock {\em Sobolev spaces, their generalizations and elliptic problems in
  smooth and Lipschitz domains}.
\newblock Springer, 2015.

\bibitem{arcangeli2007extension}
R.~Arcang{\'e}li, M.~C. L{\'o}pez~de Silanes, and J.~J. Torrens.
\newblock An extension of a bound for functions in {S}obolev spaces, with
  applications to ($m$, $s$)-spline interpolation and smoothing.
\newblock {\em Numerische Mathematik}, 107(2):181--211, 2007.

\bibitem{brezis2018gagliardo}
H.~Brezis and P.~Mironescu.
\newblock {G}agliardo--{N}irenberg inequalities and non-inequalities: the full
  story.
\newblock In {\em Annales de l'Institut Henri Poincar{\'e} C, Analyse non
  lin{\'e}aire}, volume~35, pages 1355--1376. Elsevier, 2018.

\bibitem{butzer1969fundamental}
P.~Butzer and K.~Scherer.
\newblock On the fundamental approximation theorems of {D.} {J}ackson, {SN}
  {B}ernstein and theorems of {M.} {Z}amansky and {SB} {S}te{\v{c}}kin.
\newblock {\em aequationes mathematicae}, 3:170--185, 1969.

\bibitem{cheung2018convergence}
K.~C. Cheung, L.~Ling, and R.~Schaback.
\newblock ${H}^2$-convergence of least-squares kernel collocation methods.
\newblock {\em SIAM Journal on Numerical Analysis}, 56(1):614--633, 2018.

\bibitem{marchi2005optimal}
S.~De~Marchi, R.~Schaback, and H.~Wendland.
\newblock Near-optimal data-independent point locations for radial basis
  function interpolation.
\newblock {\em Advances in Computational Mathematics}, 23(3):317--330, 2005.

\bibitem{devore1973saturation}
R.~DeVore and F.~Richards.
\newblock Saturation and inverse theorems for spline approximation.
\newblock In {\em Spline Functions and Approximation Theory: Proceedings of the
  Symposium held at the University of Alberta, Edmonton May 29 to June 1,
  1972}, pages 73--82. Springer, 1973.

\bibitem{duchon1978erreur}
J.~Duchon.
\newblock Sur l{\textquoteright}erreur d{\textquoteright}interpolation des
  fonctions de plusieurs variables par les ${D}^m$-splines.
\newblock {\em RAIRO. Analyse num\'erique}, 12(4):325--334, 1978.

\bibitem{fasshauer2007meshfree}
G.~E. Fasshauer.
\newblock {\em Meshfree {A}pproximation {M}ethods with {MATLAB}}, volume~6.
\newblock World Scientific, 2007.

\bibitem{fasshauer2015kernel}
G.~E. Fasshauer and M.~J. McCourt.
\newblock {\em Kernel-based {A}pproximation {M}ethods using {MATLAB}},
  volume~19.
\newblock World Scientific Publishing Company, 2015.

\bibitem{haas2023mind}
M.~Haas, D.~Holzm\"{u}ller, U.~Luxburg, and I.~Steinwart.
\newblock Mind the spikes: Benign overfitting of kernels and neural networks in
  fixed dimension.
\newblock In A.~Oh, T.~Naumann, A.~Globerson, K.~Saenko, M.~Hardt, and
  S.~Levine, editors, {\em Advances in Neural Information Processing Systems},
  volume~36, pages 20763--20826. Curran Associates, Inc., 2023.

\bibitem{hangelbroek2018inverse}
T.~Hangelbroek, F.~Narcowich, C.~Rieger, and J.~Ward.
\newblock An inverse theorem for compact {L}ipschitz regions in $\mathbb{R}^d$
  using localized kernel bases.
\newblock {\em Mathematics of Computation}, 87(312):1949--1989, 2018.

\bibitem{hangelbroek2018direct}
T.~Hangelbroek, F.~J. Narcowich, C.~Rieger, and J.~D. Ward.
\newblock Direct and inverse results on bounded domains for meshless methods
  via localized bases on manifolds.
\newblock {\em Contemporary Computational Mathematics-A Celebration of the 80th
  Birthday of Ian Sloan}, pages 517--543, 2018.

\bibitem{lin2024kernel}
S.-B. Lin, X.~Chang, and X.~Sun.
\newblock Kernel interpolation of high dimensional scattered data.
\newblock {\em SIAM Journal on Numerical Analysis}, 62(3):1098--1118, 2024.

\bibitem{mhaskar2010bernstein}
H.~Mhaskar, F.~Narcowich, J.~Prestin, and J.~Ward.
\newblock ${L}^p$ {B}ernstein estimates and approximation by spherical basis
  functions.
\newblock {\em Mathematics of Computation}, 79(271):1647--1679, 2010.

\bibitem{mirzaei2018petrov}
D.~Mirzaei.
\newblock A {P}etrov--{G}alerkin kernel approximation on the sphere.
\newblock {\em SIAM Journal on Numerical Analysis}, 56(1):274--295, 2018.

\bibitem{mueller2009komplexitaet}
S.~M{\"u}ller.
\newblock {\em Komplexit{\"a}t und Stabilit{\"a}t von kernbasierten
  Rekonstruktionsmethoden (Complexity and Stability of Kernel-based
  Reconstructions)}.
\newblock PhD thesis, Fakult{\"a}t f{\"u}r Mathematik und Informatik,
  Georg-August-Universit{\"a}t G{\"o}ttingen, 2009.

\bibitem{narcowich2006sobolev}
F.~Narcowich, J.~Ward, and H.~Wendland.
\newblock {S}obolev {E}rror {E}stimates and a {B}ernstein {I}nequality for
  {S}cattered {D}ata {I}nterpolation via {R}adial {B}asis {F}unctions.
\newblock {\em Constructive Approximation}, 24(2):175--186, 2006.

\bibitem{narcowich2007direct}
F.~J. Narcowich, X.~Sun, J.~D. Ward, and H.~Wendland.
\newblock {D}irect and {I}nverse {S}obolev {E}rror {E}stimates for {S}cattered
  {D}ata {I}nterpolation via {S}pherical {B}asis {F}unctions.
\newblock {\em Foundations of Computational Mathematics}, 7:369--390, 2007.

\bibitem{rieger2008sampling}
C.~Rieger.
\newblock {\em Sampling inequalities and applications}.
\newblock PhD thesis, Universität Göttingen, 2008.

\bibitem{rieger2010sampling}
C.~Rieger and B.~Zwicknagl.
\newblock Sampling inequalities for infinitely smooth functions, with
  applications to interpolation and machine learning.
\newblock {\em Advances in Computational Mathematics}, 32(1):103, 2010.

\bibitem{rychkov1999restrictions}
V.~S. Rychkov.
\newblock On restrictions and extensions of the {B}esov and
  {T}riebel--{L}izorkin spaces with respect to {L}ipschitz domains.
\newblock {\em Journal of the London Mathematical Society}, 60(1):237--257,
  1999.

\bibitem{santin2024optimality}
G.~Santin, T.~Wenzel, and B.~Haasdonk.
\newblock On the optimality of target-data-dependent kernel greedy
  interpolation in {S}obolev {R}eproducing {K}ernel {H}ilbert {S}paces.
\newblock {\em SIAM Journal on Numerical Analysis}, 62(5):2249--2275, 2024.

\bibitem{schaback1995error}
R.~Schaback.
\newblock Error estimates and condition numbers for radial basis function
  interpolation.
\newblock {\em Advances in Computational Mathematics}, 3(3):251--264, 1995.

\bibitem{schaback2018superconvergence}
R.~Schaback.
\newblock Superconvergence of kernel-based interpolation.
\newblock {\em Journal of Approximation Theory}, 235:1 -- 19, 2018.

\bibitem{schaback2002inverse}
R.~Schaback and H.~Wendland.
\newblock Inverse and saturation theorems for radial basis function
  interpolation.
\newblock {\em Mathematics of Computation}, 71(238):669--681, 2002.

\bibitem{steinwart2008support}
I.~Steinwart and A.~Christmann.
\newblock {\em Support vector machines}.
\newblock Springer Science \& Business Media, 2008.

\bibitem{temlyakov2018marcinkiewicz}
V.~Temlyakov.
\newblock The {M}arcinkiewicz-type discretization theorems.
\newblock {\em Constructive Approximation}, 48:337--369, 2018.

\bibitem{triebel2002function}
H.~Triebel.
\newblock {F}unction spaces in {L}ipschitz domains and on {L}ipschitz
  manifolds. {C}haracteristic functions as pointwise multipliers.
\newblock {\em Revista Matem{\'a}tica Complutense}, 15(2):475--524, 2002.

\bibitem{ward2012lp}
J.~P. Ward.
\newblock ${L}^p$ {B}ernstein inequalities and inverse theorems for {RBF}
  approximation on $\mathbb{R}^d$.
\newblock {\em Journal of Approximation Theory}, 164(12):1577--1593, 2012.

\bibitem{wendland2005scattered}
H.~Wendland.
\newblock {\em Scattered {D}ata {A}pproximation}, volume~17 of {\em Cambridge
  Monographs on Applied and Computational Mathematics}.
\newblock Cambridge University Press, Cambridge, 2005.

\bibitem{wendland2005approximate}
H.~Wendland and C.~Rieger.
\newblock Approximate interpolation with applications to selecting smoothing
  parameters.
\newblock {\em Numerische Mathematik}, 101(4):729--748, 2005.

\bibitem{wenzel2021novel}
T.~Wenzel, G.~Santin, and B.~Haasdonk.
\newblock A novel class of stabilized greedy kernel approximation algorithms:
  Convergence, stability and uniform point distribution.
\newblock {\em Journal of Approximation Theory}, 262:105508, 2021.

\bibitem{wenzel2023analysis}
T.~Wenzel, G.~Santin, and B.~Haasdonk.
\newblock Analysis of target data-dependent greedy kernel algorithms:
  Convergence rates for $f$-, $f${\textperiodcentered} {$P$}-and
  $f$/{$P$}-greedy.
\newblock {\em Constructive Approximation}, 57(1):45--74, 2023.

\bibitem{wenzel2024stability}
T.~Wenzel, G.~Santin, and B.~Haasdonk.
\newblock Stability of convergence rates: kernel interpolation on
  non-{L}ipschitz domains.
\newblock {\em IMA Journal of Numerical Analysis}, 44(3), 2024.

\end{thebibliography}
}
\bibliographystyle{abbrv}

\appendix 

\section{Further proofs} \label{sec:proofs}

\subsection{Proof of \Cref{th:utility_theorem_new}}

The proof of \Cref{th:utility_theorem_new} is constructive, but a bit technical.
For enhanced readability, the most important steps of the proof are visualized in \Cref{fig:vis_proof}.

\begin{figure}[t]
\setlength\fwidth{.7\textwidth}
\input{Figures/vis_proof_1.tex}
\input{Figures/vis_proof_2.tex}		%

\input{Figures/vis_proof_3.tex}
\input{Figures/vis_proof_4.tex}
\caption{Visualization of the key steps of the proof of \Cref{th:inverse_statement_generalized}.
For the sake of a helpful visualization, $q=0.2$ was chosen, despite it does not match the assumptions. \\
\textbf{Top left}: Initial set $Y_N$ (gray), reference set $Y_0$ (red) with circles of radius $\frac{1}{2} q = 0.1$ and set $Y_1'$ (purple, see Eq.~\eqref{eq:definition_y'}). 
\textbf{Top right}: Set $Y_1''$ (purple, see Eq.~\eqref{eq:definition_y''}) and clusters of points refering to the sets $T_i^{(N)}$ (gray, see Eq.~\eqref{eq:proof:definition_TiN}).
\textbf{Bottom left}: Restriction of the sets $T_i^{(N)}$ to $\tilde{T}_i^{(N)}$, see Eq.~\eqref{eq:proof:definition_TiN_tilde}. 
\textbf{Bottom right}: Additional visualization of $Y_1$ (red crosses, see Eq.~\eqref{eq:set_Y1}).}
\label{fig:vis_proof}
\end{figure}

\begin{proof}[Proof of \Cref{th:utility_theorem_new}]
As a first step, we approximate the integral $\int_\Omega |g(y)|^2 ~ \mathrm{d}y = \Vert g \Vert_{L^2(\Omega)}$ with a quadrature:
Let $(y_j)_{j \in \N}$ be the sequence of geometric greedy points in $\Omega \subset \R^d$, $Y_N := \{y_1, ..., y_N\}$. 
Let $N \in \N$:
\begin{itemize}
\item We cover $\Omega$ with patches $\{ E_i^{(N)} \}_{i=1}^N$, 
where $E_j^{(N)} \subset \Omega$ consists of all the points that are closer to $y_j$ than to any other $y_i, i=1, ..., N, i \neq j$:
\begin{align*}
E_j^{(N)} := \{y \in \Omega ~:~ \Vert y - y_j \Vert_2 \leq \Vert y - y_i \Vert_2 \quad \forall i \neq j \}.
\end{align*}
Like this it follows that $\bigcup_{i=1}^N E_i^{(N)} = \Omega$ and $\mu(E_i^{(N)} \cap E_j^{(N)}) = 0$ for all $i \neq j$.
Furthermore all the patches $E_i^{(N)}, i=1, ..., N$ are asymptotically of the same size:
On the one hand, using that for any $i=1, ..., N$, 
the patch $E_i^{(N)}$ is included in a ball of radius $h_{Y_N, \Omega}$ around $y_i$, it holds
\begin{align}
\label{eq:patch_size_upper_bound}
\mu(E_i^{(N)}) \leq C_d h_{Y_N, \Omega}^d \stackrel{\text{Eq.\ } \eqref{eq:geometric_uniformity}}{\leq} C_d 2^d q_{Y_N}^d \stackrel{\text{Eq.\ } \eqref{eq:geometric_bounds}}{\leq} 2^d C_d C_\Omega^d N^{-1}.
\end{align}
On the other hand, for any $i=1, ..., N$, 
the cone $C(y_i, \xi(y_i), \alpha, q_{Y_N})$ associated to $y_i$ is contained in $E_i^{(N)}$, thus
\begin{align}
\label{eq:patch_size_lower_bound}
\mu(E_i^{(N)}) \geq C_{d, \alpha} q_{Y_N}^d \stackrel{\text{Eq.\ } \eqref{eq:geometric_uniformity}}{\geq} 
C_{d, \alpha} 2^{-d} h_{Y_N, \Omega}^d 
\stackrel{\text{Eq.\ } \eqref{eq:geometric_bounds}}{\geq}  2^{-d} c_\Omega^{d} C_{d, \alpha} N^{-1}.
\end{align}
\item For $N \rightarrow \infty$ it holds $\max_{i=1, ..., N} \sup_{a, b \in \overline{E_i^{(N)}}}( \Vert a - b \Vert_2) \rightarrow 0$, 
therefore for any integrable function $g \in \mathcal{C}(\Omega)$ ($g \in \mathcal{C}(\Omega)$ bounded provides integrability) the definition of the Rieman integral yields
\begin{align*}
\int_{\Omega} |g(y)|^2 ~ \mathrm{d}y = \lim_{N \rightarrow \infty} \sum_{i=1}^N |g(y_i)|^2 \cdot \mu(E_i^{(N)}).
\end{align*}
In particular we have
\begin{alignat*}{2}
&\forall \varepsilon > 0 ~ \exists N_\varepsilon \in \N ~ \forall N \geq N_\epsilon && \quad \left| \int_{\Omega} |g(y)|^2 ~ \mathrm{d}y - \sum_{i=1}^N |g(y_i)|^2 \cdot \mu(E_i^{(n)}) \right| < \varepsilon \\
\Rightarrow \quad &\exists \tilde{N} \in \N ~ \forall N \geq \tilde{N}&& \quad \sum_{i=1}^N |g(y_i)|^2 \cdot \mu(E_i^{(n)}) < 2 \cdot \int_{\Omega} |g(y)|^2 ~ \mathrm{d}y.
\end{alignat*}
Using Eq.\ \eqref{eq:patch_size_lower_bound} we obtain finally
\begin{align}
\label{eq:desired_estimate}
\forall N > \tilde{N}& \quad \frac{c_\Omega^{d} C_{d, \alpha}}{2^{d}} \cdot \frac{1}{N} \cdot \sum_{i=1}^N |g(y_i)|^2 < 2 \cdot \int_{\Omega} |g(y)|^2 ~ \mathrm{d}y \notag \\
\Leftrightarrow \forall N > \tilde{N}& \quad \frac{1}{N} \cdot \sum_{i=1}^N |g(y_i)|^2 < \frac{2^{d+1}}{c_\Omega^{d} C_{d, \alpha}} \cdot \int_{\Omega} |g(y)|^2 ~ \mathrm{d}y.
\end{align}
\end{itemize}

The set $Y_N$ of points can be quite huge as $N$ is not further specified.
Thus, as a next step we want to construct smaller sets of points that still satisfy an inequality as in Eq.\ \eqref{eq:desired_estimate}.
This is done by selecting a suitable subset of $Y_N$ (namely the set $Y_1$, see Eq.~\eqref{eq:set_Y1}): \\
For $q \leq \min \left( \frac{2^{-1/d} c_\Omega}{7 C_\Omega} \cdot q_{Y_0}, \frac{2}{5}r \right) \leq \frac{c_\Omega}{7 C_\Omega} \cdot q_{Y_0} < \frac{1}{7} \cdot q_{Y_0}$, we define 
\begin{align}
\label{eq:definition_y'}
Y'_1 := \{ y_1, y_2, ..., y_{n^*}\}
\end{align}
with $n^* \geq 2$ the smallest integer such that $h_{Y_1' \cup \{ y_{n* + 1} \}, \Omega} < q$.
From this definition we conclude
\begin{align}
\label{eq:proof:geom_props_1}
\frac{1}{2} q \leq \frac{1}{2} h_{Y'_1, \Omega} \leq h_{Y'_1 \cup \{ y_{n^*+1} \}, \Omega} < q
\end{align}
where the middle inequality is because of Eq.~\eqref{eq:geometric_uniformity}.
For the separation distance $q_{Y_1'}$ we have due to Eq.~\eqref{eq:geometric_uniformity}
\begin{align}
\label{eq:proof:geom_props_2}
\frac{1}{2} q \leq \frac{1}{2} h_{Y_1', \Omega} \leq q_{Y_1'} \leq h_{Y_1', \Omega} < 2q.
\end{align}
With this, we define $Y''_1 \subset Y'_1$ as the subset of all the points which are not too close to $Y_0$, i.e.\
\begin{align}
\label{eq:definition_y''}
Y''_1 := Y'_1 \cap \left( \Omega \setminus \bigcup_{y \in Y_0} B_{\frac{1}{2} q}(y) \right),
\end{align}
where $B_r(y) := \{ z \in \R^d ~ | ~ \Vert z - y \Vert_2 < r \}$, i.e.\ open balls. 

The set $Y''_1$ is non-empty and well distributed:
\begin{itemize}
\item The set $Y_1''$ consists of at least two points:
Let $y \neq y' \in Y_0$ be the closest neighboring points of $Y_0$, i.e.\ $\Vert y - y' \Vert_2 = 2q_{Y_0} > 14q$, 
and thus $\mathrm{dist}(B_{\frac{1}{2}q}(y), B_{\frac{1}{2}q}(y')) > 13q$.
On the other hand, $h_{Y_1', \Omega} < 2q$, i.e.\ there must be at least two points in $Y_1''$ (namely ``in between'' $y$ and $y'$ as $\Omega$ is connected).
\item For the separation distance $q_{Y''_1}$ we simply have by Eq.~\eqref{eq:proof:geom_props_2}
\begin{align*}
q_{Y''_1} \geq q_{Y'_1} \geq \frac{1}{2} q.
\end{align*}
\item For the fill distance $h_{Y''_1, \Omega}$ we calculate:
\begin{align*}
h_{Y''_1, \Omega} \equiv& \sup_{y \in \Omega} \min_{y'' \in Y''_1} \Vert y - y'' \Vert_2 \\
=& \sup_{y \in \Omega} \min_{y'' \in Y''_1} \min_{y' \in Y'_1} \Vert y - y' + y' - y'' \Vert_2 \\
\leq& \sup_{y \in \Omega} \min_{y'' \in Y''_1} \min_{y' \in Y'_1} \Vert y - y' \Vert_2 + \Vert y' - y'' \Vert_2.
\end{align*}
Now pick $\tilde{y}' \in Y'_1$ as the closest $y'$ to $y$, i.e.\ $\tilde{y}' := \tilde{y}'(y) := \argmin_{y' \in Y'_1} \Vert y - y' \Vert_2$.
Then we continue to estimate:
\begin{align*}
h_{Y''_1, \Omega} \leq& \sup_{y \in \Omega} \left( \min_{y'' \in Y''_1} \Vert y - \tilde{y}'(y) \Vert_2 + \Vert \tilde{y}'(y) - y'' \Vert_2 \right) \\
\leq& \sup_{y \in \Omega} \left( \Vert y - \tilde{y}'(y) \Vert_2 + \min_{y'' \in Y''_1}  \Vert \tilde{y}'(y) - y'' \Vert_2 \right) \\
\leq& \sup_{y \in \Omega} \Vert y - \tilde{y}'(y) \Vert_2 + \sup_{y \in \Omega} \min_{y'' \in Y''_1}  \Vert \tilde{y}'(y) - y'' \Vert_2 \\
\leq& h_{Y'_1, \Omega} + \sup_{y \in \Omega} \min_{y'' \in Y''_1}  \Vert \tilde{y}'(y) - y'' \Vert_2.
\end{align*}
If $\tilde{y}'(y) \in Y_1''$, then the second summand is obviously zero.
If $\tilde{y}'(y) \in Y_1' \setminus Y_1''$, then by Eq.\ \eqref{eq:definition_y''} there is a point $y_0(y) \in Y_0$ such that $\Vert \tilde{y}'(y) - y_0(y) \Vert_2 < \frac{1}{2} q$. 
Then we choose any point $\tilde{y} := \tilde{y}(y) \in \Omega$ with $\Vert y_0(y) - \tilde{y}(y) \Vert_2 = \frac{5}{2} q$.
Such a point exists as $\frac{5}{2} q \leq r$, i.e. the distance is smaller than the radius of the cone condition of $\Omega$.
Adding $0 = - y_0(y) + y_0(y) - \tilde{y}(y) + \tilde{y}(y)$ gives
\begin{align*}
\sup_{y \in \Omega} \min_{y'' \in Y''_1}  \Vert \tilde{y}'(y) - y'' \Vert_2 
&= \sup_{y \in \Omega} \min_{y'' \in Y''_1}  \Vert \tilde{y}'(y) - y_0(y) + y_0(y) - \tilde{y}(y) + \tilde{y}(y) - y'' \Vert_2 \\
&\leq \sup_{y \in \Omega} \min_{y'' \in Y''_1}  \Vert \tilde{y}'(y) - y_0(y) \Vert_2 + \Vert y_0(y) - \tilde{y}(y) \Vert_2+ \Vert \tilde{y}(y) - y'' \Vert_2 \\
&\leq \sup_{y \in \Omega} \Vert \tilde{y}'(y) - y_0(y) \Vert_2 + \sup_{y \in \Omega} \Vert y_0(y) - \tilde{y}(y) \Vert_2 + \sup_{y \in \Omega} \min_{y'' \in Y''_1} \Vert \tilde{y}(y) - y'' \Vert_2 \\
&\leq \frac{1}{2} q + \frac{5}{2} q + 2q = 5q.
\end{align*}
In the last step we used for the last term that the closest point $y^*(y) \in Y'_1$ to $\tilde{y}(y)$ has a distance of at most $2q$ (due to $h_{Y'_1, \Omega} < 2q$) and that this point must actually belong to $Y_1''$ 
(because $\Vert y_0(y) - y^*(y) \Vert_2 = \Vert y_0(y) - \tilde{y}(y) + \tilde{y}(y) - y^*(y) \Vert_2 \geq \left| \Vert y_0(y) - \tilde{y}(y) \Vert - \Vert \tilde{y}(y) - y^*(y) \right| > |\frac{5}{2} q - 2q | = \frac{1}{2} q$). \\
Thus, finally we obtain
\begin{align*}
h_{Y_1'', \Omega} \leq 5q + 2q = 7q.
\end{align*}
\end{itemize}
For ease of notation, we set $Y_1'' =: \{y''_1, ..., y''_{|Y_1''|} \}$.
Now we create sets $T_i^{(N)}$ by assigning nearby points from $Y_N$ (for $N$ large enough) to the points from $Y_1''$:
\begin{align}
\label{eq:proof:definition_TiN}
T_i^{(N)} :=& \{ z \in Y_N ~ | ~ \Vert y''_i - z \Vert_2 \leq \frac{1}{3} q \}, \qquad &&i=1, ..., |Y_1''|.
\end{align}
The distance $\frac{1}{3}q$ was chosen such that $T_i^{(N)} \cap T_j^{(N)} = \emptyset$ for any pairwise distinct $i, j = 1, ..., |Y_1''|$. 
For convenience we introduce
\begin{align*}
T_i^{(N)} =: \left\{ y_1^{(i, N)}, ..., y_{|T_i^{(N)}|}^{(i, N)} \right\}.
\end{align*}
Now we prove that for $N \rightarrow \infty$, all the sets $T_i^{(N)}$, $i=1, ..., |Y_1''|$ are of asymptotically equal size:
\begin{itemize}
\item The number of points in $T_i^{(N)}$ must be large enough, to at least allow to cover the surrounding of $y_i'' \in Y_1''$ in terms of the volume:
Due to the assumption on the cone condition, 
the surrounding of $y_i''$ consist at least of the associated cone $C(y_i'', \xi(y_i''), \alpha, \frac{1}{3} q)$:
\begin{align*}
\bigcup_{y \in T_i^{(N)}} B_{h_{Y_N, \Omega}}(y) &\supset C(y_i'', \xi(y_i''), \alpha, \frac{1}{3}q) \\
\Rightarrow \quad \left|T_i^{(N)} \right| \cdot C_d h_{Y_N, \Omega}^{d} &\geq \mu(C(y_i'', \xi(y_i''), \alpha, \frac{1}{3}q)) = 3^{-d} C_{d, \alpha} \cdot q^d.
\end{align*}
Rearranging and using the geometric bounds from Eq.\ \eqref{eq:geometric_bounds} and Eq.\ \eqref{eq:geometric_uniformity} gives
\begin{align}
\label{eq:subset_size_lower_bound}
\left| T_i^{(N)} \right| &\geq \frac{3^{-d} C_{d, \alpha}}{C_d} \cdot \frac{q^d}{h_{Y_N, \Omega}^d}
\geq \frac{3^{-d} C_{d, \alpha}}{C_d} \cdot \frac{q^d}{2^d C_\Omega^d N^{-1}} \notag \\
&= \frac{C_{d, \alpha} q^d}{6^d C_d C_\Omega^d} \cdot N.
\end{align}
\item Using again a volume comparison argument, we have (for $N > n$ large enough such that $q_{Y_N} \leq \frac{2}{3} q$):
\begin{align*}
\bigcup_{y \in T_i^{(N)}} B_{q_{Y_N}}(y) &\subset B_{\frac{1}{3}q + q_{Y_N}}(y''_i) \\
\Rightarrow \quad \left| T_i^{(N)} \right| \cdot C_d q_{Y_N}^d &\leq C_d \cdot \left( \frac{1}{3}q + q_{Y_N} \right)^d \leq C_d q^d.
\end{align*}
Using the geometric bounds from Eq.\ \eqref{eq:geometric_bounds} and the asymptotical uniform distribution of Eq.\ \eqref{eq:geometric_uniformity} we arrive at
\begin{align}
\label{eq:subset_size_upper_bound}
\left| T_i^{(N)} \right| &\leq q^d \cdot q_{Y_N}^{-d} \leq q^d 2^d c_\Omega^d N.
\end{align}
\item Combining Eq.\ \eqref{eq:subset_size_lower_bound} and Eq.\ \eqref{eq:subset_size_upper_bound} formalizes our claim that the sets $T_j^{(N)}$ are of asymptotical equal size, 
because the lower and upper bound do actually not depend on $i = 1, ..., |Y_1''|$:
\begin{align}
\label{eq:subset_size_asymp}
\frac{C_{d, \alpha}}{6^d C_d C_\Omega^d} \cdot q^d N \leq 
\left| T_j^{(N)} \right| 
\leq 2^d c_\Omega^d \cdot q^d N.
\end{align}
\end{itemize}
Now we define $\tilde{T}_i^{(N)} \subset T_i^{(N)}$ for $i=1, ..., |Y_1''|$ be restricting the sets to the same amount of points, $m := \min_{i=1, ..., |Y_1''|} |T_i^{(N)}|$:
\begin{align}
\label{eq:proof:definition_TiN_tilde}
\tilde{T}_i^{(N)} := \left\{ y_1^{(i, N)}, ..., y_m^{(i, N)} \right\}.
\end{align}
We remark that the sets $\tilde{T}_i^{(N)}$ depend on the ordering of the points $\{ y_1^{(i, N)}, ..., y_{|T_i^{(N)}|}^{(i, N)} \}$, however it does not matter for the following which one is chosen. \\
Now we want to show that the number of points contained in 
\begin{align*}
\tilde{T}^{(N)} := \bigcup_{i=1}^{|Y_1''|} \tilde{T}_i^{(N)}
\end{align*}
is comparable to the number of points in the set $Y_N$ (which consists of $N$ points):
\begin{itemize}
\item For this we need a lower bound on $|Y_1''|$.
From Eq.\ \eqref{eq:proof:geom_props_2} we have $q_{Y_1'} \geq \frac{1}{2}q$, i.e.\ in every (open!) ball $B_{\frac{1}{2} q}(y)$ for $y \in Y_0$, there can be at most one point from $Y_1'$.
Therefore, we obtain $|Y_1''| \geq |Y_1'| - |Y_0|$.
Further using the geometric bound Eq.~\eqref{eq:geometric_bounds} and Eq.~\eqref{eq:proof:geom_props_1}, we obtain
\begin{align*}
|Y_1'| \geq c_\Omega^{d} h_{Y_1', \Omega}^{-d} \geq c_\Omega^d (2q)^{-d} = 2^{-d} c_\Omega^d q^{-d} 
\end{align*}
as well as
\begin{align*}
|Y_0| \leq C_\Omega^d q_{Y_0}^{-d} \leq C_\Omega^d \frac{1}{2} \frac{c_\Omega^d}{7^d C_\Omega^d} q^{-d} 
= \frac{1}{2} \frac{c_\Omega^d}{7^d} q^{-d}.
\end{align*}
Thus we have
\begin{align}
\label{eq:lower_bound_Y1''}
|Y_1''| \geq |Y_1'| - |Y_0| \geq& ~ \frac{c_\Omega^d}{2^d} q^{-d} - \frac{1}{2} \frac{c_\Omega^d}{7^d} q^{-d} \notag \\
>& \left( \frac{1}{2} \frac{c_\Omega^d}{2^d} \right) q^{-d}.
\end{align}
Now we can finally estimate the ratio of points contained in $\tilde{T}^{(N)}$ by putting together Eq.~\eqref{eq:subset_size_lower_bound} and \eqref{eq:lower_bound_Y1''}:
\begin{align}
\label{eq:subset_size_ratio}
\left| \tilde{T}^{(N)} \right| = \left| \bigcup_{i=1}^{|Y_1''|} \tilde{T}_i^{(N)} \right| \geq& ~ |Y_1''| \cdot \min_{i=1, ..., |Y_1''|} | \tilde{T}_i^{(N)} | \notag \\
\geq& \left( \frac{1}{2} \frac{c_\Omega^d}{2^d} \right) q^{-d} \cdot \frac{C_{d, \alpha}}{6^d C_d C_\Omega^d} \cdot q^d N \notag \\
=& \left( \frac{1}{2} \cdot \frac{c_\Omega^d C_{d, \alpha}}{12^d C_d C_\Omega^d} \right)  N, \notag \\
\Leftrightarrow \qquad \frac{N}{\left| \tilde{T}^{(N)} \right|} \leq& ~ 2 \cdot \frac{12^d C_d C_\Omega^d}{c_\Omega^d C_{d, \alpha}}.
\end{align}
\end{itemize}
Now we are ready to derive our main statement:
Recalling Eq.~\eqref{eq:proof:definition_TiN_tilde},
we consider the sum $\sum_{y \in \tilde{T}^{(N)}} |g(y)|$ (consisting of $|\tilde{T}^{(N)}| = |Y_1''| \cdot \min_{i=1, ..., |Y_1''|} |T_i^{(N)}|$ summands)
and split it into $m \equiv \min_{i=1, ..., |Y_1''|} |T_i^{(N)}|$ summands, such that every summand uses exactly one point from every $\tilde{T}_i^{(N)}$:
\begin{align*}
\frac{1}{|\tilde{T}^{(N)}|} \sum_{y \in \tilde{T}^{(N)}} |g(y)|^2
= \frac{1}{m} \sum_{\ell = 1}^m \left( \frac{1}{|Y_1''|} \sum_{i=1}^{|Y_1''|} |g(y_{\ell}^{(i, N)})|^2 \right).
\end{align*}
Now we have
\begin{align*}
\frac{1}{m} \sum_{\ell = 1}^m \left( \frac{1}{|Y_1''|} \sum_{i=1}^{|Y_1''|} |g(y_{\ell}^{(i, N)})|^2 \right) 
&~~~=~~~ \frac{1}{|\tilde{T}^{(N)}|} \sum_{y \in \tilde{T}^{(N)}} |g(y)|^2 \\
&~~~\leq~~~ \frac{N}{|\tilde{T}^{(N)}|} \cdot \frac{1}{N} \cdot \sum_{y \in Y_N} |g(y)|^2 \\
&\stackrel{\text{Eq.\ }\eqref{eq:desired_estimate}}{\leq} \frac{N}{|\tilde{T}^{(N)}|} \cdot \frac{2^{d+1}}{c_\Omega^{d} C_{d, \alpha}} \cdot \int_{\Omega} |g(y)|^2 ~ \mathrm{d}y \\ 
&\stackrel{\text{Eq.\ }\eqref{eq:subset_size_ratio}}{\leq} 2 \cdot \frac{12^d C_d C_\Omega^d}{c_\Omega^d C_{d, \alpha}} \cdot \frac{2^{d+1}}{c_\Omega^{d} C_{d, \alpha}} \cdot \int_{\Omega} |g(y)|^2 ~ \mathrm{d}y \\
&~~~=:~~~ \tilde{C}_{d, \alpha} \cdot \int_{\Omega} |g(y)|^2 ~ \mathrm{d}y.
\end{align*}
From here we can conclude that there exists at least one $\ell^* \in \{ 1, ..., m \}$ such that
\begin{align}
\label{eq:desired_estimate_2}
\frac{1}{|Y_1''|} \sum_{i=1}^{|Y_1''|} |g(y_{\ell^*}^{(i, N)})|^2 \leq \tilde{C}_{d, \alpha} \cdot \int_{\Omega} |g(y)|^2 ~ \mathrm{d}y,
\end{align}
which is exactly the inequality Eq~\eqref{eq:utility_theorem_new} to be proven.
Finally we are left with analyzing the geometric properties of this set 
\begin{align}
\label{eq:set_Y1}
Y_1 := \{ y_{\ell^*}^{(i,N)} \}_{i=1}^{|Y_1''|} = \{ y_{\ell^*}^{(i,N)} \}_{i=1}^{|Y_1|}.
\end{align}

\begin{itemize}
\item For the fill distance $h_{Y_1, \Omega}$ we calculate
\begin{align}
\label{eq:proof:subset_upper_bound_fill_dist}
h_{Y_1, \Omega} &\equiv \sup_{y \in \Omega} \min_{i=1, ..., |Y_1|} \Vert y - y_{\ell^*}^{(i, N)} \Vert_2 \notag \\
&\leq \sup_{y \in \Omega} \min_{i=1, ..., |Y_1|} \Vert y - \argmin_{y_j'' \in Y_1''}(\Vert y - y_j'' \Vert_2) + \argmin_{y_j'' \in Y_1''}(\Vert y - y_j'' \Vert_2) - y_{\ell^*}^{(i,N)} \Vert_2 \notag \\
&\leq \sup_{y \in \Omega} \min_{i=1, ..., |Y_1|} \Vert y - \argmin_{y_j'' \in Y_1''}(\Vert y - y_j'' \Vert_2) \Vert_2 + \Vert \argmin_{y_j'' \in Y_n}(\Vert y - y_j'' \Vert_2) - y_{\ell^*}^{(i,N)} \Vert_2 \notag \\
&\leq \sup_{y \in \Omega} \Vert y - \argmin_{y_j'' \in Y_1''}(\Vert y - y_j'' \Vert_2) \Vert_2 + \sup_{y \in \Omega} \min_{i=1, ..., |Y_1''|} \Vert \argmin_{y_j'' \in Y_1''}(\Vert y - y_j'' \Vert_2) - y_{\ell^*}^{(i,N)} \Vert_2 \notag \\
&\leq h_{Y_1'', \Omega} + \frac{1}{3} q 
\leq 7q + \frac{1}{3}q = \frac{22}{3}q,
\end{align}
where the definition of the sets $T_i^{(n, N)}$ from Eq.\ \eqref{eq:proof:definition_TiN} was used for the final step.
\item For the separation distance $q_{Y_1}$ it holds
\begin{align*}
q_{Y_1} &\equiv \min_{\substack{i \neq j \\ i, j = 1, ..., |Y_1|}} \Vert y_{\ell^*}^{(i,N)} - y_{\ell^*}^{(j,N)} \Vert_2 \\
&= \min_{\substack{i \neq j \\ i, j = 1, ..., |Y_1|}} \left \Vert y_{\ell^*}^{(i,N)} - y''_i + y''_i - y''_j + y''_j - y_{\ell^*}^{(j,N)} \right \Vert_2 \\
&\geq \min_{\substack{i \neq j \\ i, j = 1, ..., |Y_1|}} \left| \left \Vert y_i'' - y_j'' \right \Vert_2 - \left \Vert y_{\ell^*}^{(i,N)} - y''_i + y''_j - y_{\ell^*}^{(j,N)} \right \Vert_2 \right|.
\end{align*}
The first term can be lower bounded by $2 \cdot q_{Y_1''} \geq q$ due to the definition of the separation distance of $Y_1''$, 
while the latter term can be upper bounded (after applying the triangle inequality) by $\frac{1}{3} q + \frac{1}{3} q$ due to the definition of the sets $T_i^{(N)}$ in Eq.\ \eqref{eq:proof:definition_TiN}.
Thus we finally obtain
\begin{align}
\label{eq:proof:subset_lower_bound_sep_dist}
q_{Y_1} \geq 2q_{Y_1''} - \frac{2}{3} q \geq q - \frac{2}{3}q = \frac{1}{3}q.
\end{align}
\item Combining Eq.\ \eqref{eq:proof:subset_upper_bound_fill_dist} and Eq.\ \eqref{eq:proof:subset_lower_bound_sep_dist} gives
\begin{align*}
\frac{1}{3} q \leq q_{Y_1} \leq h_{Y_1, \Omega} \leq \frac{22}{3}q,
\end{align*}
i.e.\ also the set $Y_1$ is well distributed.
\end{itemize}
Finally we want to show that also the union $Y_0 \cup Y_1$ is well distributed:
\begin{align*}
q_{Y_0 \cup Y_1} =& \min (q_{Y_0}, q_{Y_1}, \min_{y0 \in Y_0, y_1 \in Y_1} \Vert y_0 - y_1 \Vert_2) \\
\geq& \min \left(q_{Y_0}, \frac{1}{3}q, \frac{1}{2}q - \frac{1}{3}q \right) = \frac{1}{6}q,
\end{align*}
where the last quantity was estimated with Eq.~\eqref{eq:definition_y''} and Eq.~\eqref{eq:proof:definition_TiN}.
\end{proof}

\subsection{Proof of \Cref{prop:disc_L2_from_cont_L2_new}}

\begin{proof}[Proof of \Cref{prop:disc_L2_from_cont_L2_new}]
Let $(y_j)_{j \in \N}$ be the sequence of geometric greedy points in $\Omega \subset \R^d$, $Y_m := \{ y_1, ..., y_m \}$ for all $m \in \N$ as specified in \Cref{subsec:geometric_properties}.
We define $m_0 \in \{2, 3, ...\} \subset \N$ as the smallest integer such that $Y_{m_0}$ satisfies 
\begin{align*}
h_{Y_{m_0}, \Omega} < h_0 ~~ \text{and} ~~ \frac{2^{-1/d}c_\Omega}{7C_\Omega} \cdot q_{Y_{m_0}} < \frac{2}{5}r.
\end{align*}
We then define $X_0 := Y_{m_0}$. 
Then the assumption of this \Cref{prop:disc_L2_from_cont_L2_new} yields
\begin{align*}
\Vert f - s_{f, X_{0}} \Vert_{L^2(\Omega)} \leq c_f \cdot h_{X_{0}}^\mu.
\end{align*}
Define 
\begin{align}
\label{eq:definition_a}
a := \frac{3}{22} \cdot \frac{2^{-1/d} c_\Omega}{7C_\Omega} \in \left( 0, \frac{3}{154} \right) \subset (0, 1).
\end{align}
Now \Cref{th:utility_theorem_new} can be applied for $g := f - s_{f, X_{0}}$ (which is continuous and bounded),
the reference set $Y_0 := X_0$ and $q := a \cdot q_{X_{0}}$ (which satisfies $q \leq \min \left( \frac{2^{-1/d} c_\Omega}{7 C_\Omega} \cdot q_{Y_0}, \frac{2}{5}r \right)$ ), 
which yields a set $Y_1 \subset \Omega$ that satisfies
\begin{align}
\begin{aligned}
\label{eq:proof:properties_initial}
\frac{1}{3} a q_{X_{0}} \leq& ~ q_{Y_{1}} \leq h_{Y_{1}, \Omega} \leq \frac{22}{3} a  q_{X_{0}}, \\
\frac{1}{6} aq_{X_{0}} \leq& ~ q_{X_{0} \cup Y_{1}}, \\
\Vert f - s_{f, X_{0}} \Vert_{L^2(Y_{1})} \leq& ~ \sqrt{\tilde{C}_{d, \alpha}} \cdot \Vert f - s_{f, X_{0}} \Vert_{L^2(\Omega)}.
\end{aligned}
\end{align}
Now we simply define $X_1 := X_0 \cup Y_1$. Due to $s_{f, X_0}|_{X_0} = f|_{X_0}$,
Eq.~\eqref{eq:proof:properties_initial} turns into
\begin{align}
\begin{aligned}
\label{eq:proof:properties_initial2}
\frac{1}{6} a q_{X_{0}} \leq& ~ q_{X_{1}} \leq h_{X_{1}, \Omega} \leq \frac{22}{3} a  q_{X_{0}}, \\
\Vert f - s_{f, X_{0}} \Vert_{L^2(X_{1})} \leq& ~ \sqrt{\tilde{C}_{d, \alpha}} \cdot \Vert f - s_{f, X_{0}} \Vert_{L^2(\Omega)}.
\end{aligned}
\end{align}
Now the idea is to iteratively apply \Cref{th:utility_theorem_new} to build the sequence $(X_m)_{m \geq 0}$ of sets. \\ 
We formalize this strategy via induction: 
\begin{enumerate}
\item The start of the induction is given by the sets $X_{0}$ and $X_{1}$ from above, 
which satisfy the desired geometric estimates on the distribution of the points as well as the estimate between discrete $L^2(X)$ norm and continuous $L^2(\Omega)$ norm, 
see Eq.~\eqref{eq:proof:properties_initial2}.
\item Assume for $m \geq 0$ we have two sets $X_{m}, X_{m+1}$ such that 
\begin{align}
\frac{1}{6} a^{m+1} q_{X_{0}} \leq& ~ q_{X_{m+1}} \leq h_{X_{m + 1}, \Omega} \leq \frac{22}{3} a^{m+1} q_{X_{0}}, \label{eq:proof:properties_assump1} \\
\Vert f - s_{f, X_{m}} \Vert_{L^2(X_{m + 1})} \leq& ~ \sqrt{\tilde{C}_{d, \alpha}} \cdot \Vert f - s_{f, X_{m}} \Vert_{L^2(\Omega)}. \notag %
\end{align}
Now we can leverage again \Cref{th:utility_theorem_new} for $g := f - s_{f, X_{m+1}}$ and $q: = a^{m+2} q_{X_0}$ and reference set $Y_0 := X_{m+1}$, as the choice of $q$ satisfies the requirements:
\begin{align*}
q \equiv 6a \cdot \frac{1}{6} a^{m+1} q_{X_0} \stackrel{\text{Eq.}~\eqref{eq:proof:properties_assump1}}{\leq} 6a \cdot q_{X_{m+1}} \stackrel{\text{Eq.}~\eqref{eq:definition_a}}{\leq} \frac{18}{22} \cdot \frac{2^{-1/d} c_\Omega}{7C_\Omega} \cdot q_{X_{m+1}}.
\end{align*}
Thus, \Cref{th:utility_theorem_new} gives a set $Y_1$ satisfying
\begin{align}
\begin{aligned}
\label{eq:proof:properties_initial3}
\frac{1}{3} a^{m+2} q_{X_0} \leq& ~ q_{Y_1} \leq h_{Y_1, \Omega} \leq \frac{22}{3} a ^{m+2} q_{X_0}, \\
\frac{1}{6} a^{m+2} q_{X_0} \leq& ~ q_{X_{m+1} \cup Y_1}, \\
\Vert f - s_{f, X_{m+1}} \Vert_{L^2(Y_1)} \leq& ~ \sqrt{\tilde{C}_{d, \alpha}} \cdot \Vert f - s_{f, X_{m+1}} \Vert_{L^2(\Omega)}. 
\end{aligned}
\end{align}
Now we define $X_{m+2} := X_{m+1} \cup Y_1$.
Due to $s_{f, X_{m+1}}|_{X_{m+1}} = f|_{X_{m+1}}$ Eq.~\eqref{eq:proof:properties_initial3} turns into
\begin{align*}
\frac{1}{6} a^{m+2} q_{X_0} \leq& ~ q_{X_{m+2}} \leq h_{X_{m+2}, \Omega} \leq \frac{22}{3} a^{m+2} q_{X_0}, \\
\Vert f - s_{f, X_{m+1}} \Vert_{L^2(X_{m+2})} \leq& ~ \sqrt{\tilde{C}_{d, \alpha}} \cdot \Vert f - s_{f, X_{m+1}} \Vert_{L^2(\Omega)}. 
\end{align*}
Thus, we have found a new pair of sets $X_{m+1}, X_{m+2}$ with corresponding properties and the induction step is finished.
\end{enumerate}
Overall we have found a nested sequence of sets $(X_m)_{m \geq 0}$ that satisfy the geometric properties for $m \geq 0$.
To obtain the desired estimate on the discrete $L^2$ norm, we just need to leverage the assumption on the decay of $\Vert f - s_{f, X} \Vert_{L^2(\Omega)}$.
The constant $c_0$ is given as $c_0 := q_{X_0}$ and only depends on $h_0$ and the region $\Omega$.
\end{proof}

\newpage

\end{document}